\documentclass{amsart}
\usepackage{amssymb}
\usepackage{amsmath}
\usepackage{amsfonts}

\newtheorem{theorem}{Theorem}[section]
\theoremstyle{plain}

\newtheorem{definition}{Definition}[section]

\newtheorem{lemma}{Lemma}[section]

\newtheorem{proposition}{Proposition}[section]
\newtheorem{remark}{Remark}[section]

\numberwithin{equation}{section}
\input{tcilatex}

\begin{document}
\title[]{ Homogenization of the Stokes type Model in Porous media with slip
boundary conditions}
\author{Lazarus Signing}
\address{University of Ngaoundere, Department of Mathematics and Computer
Science, P.O.Box 454 Ngaoundere (Cameroon).}
\email{lsigning@yahoo.fr}
\urladdr{}
\thanks{}
\curraddr{ }
\email{ }
\urladdr{}
\thanks{}
\address{ }
\urladdr{}
\date{}
\dedicatory{}
\thanks{}

\begin{abstract}
This work is devoted to the study of the limiting behaviour of the Stokes
type fluid flows in porous media. The boundary conditions here are of the
Fourier-Neumann's type on the boundary of the holes. Under the periodic
hypothesis on the structure of the medium and on the coefficients of the
viscosity tensor, one convergence result is proved. Our approach is the well
known two-scale convergence method.
\end{abstract}

\maketitle

\section{Introduction}

Let $\Omega $ be a smooth bounded open set of $\mathbb{R}_{x}^{N}$ (the $N$%
-dimensional numerical space of variables $x=\left( x_{1},...,x_{N}\right) $%
) with $N\geq 2$. We consider a compact subset $T$ of $\mathbb{R}_{y}^{N}$
with smooth boundary and nonempty interior such that 
\begin{equation}
T\subset Y=\left( -\frac{1}{2},\frac{1}{2}\right) ^{N}\text{.}  \label{eq1.1}
\end{equation}%
For any $\varepsilon >0$, we define 
\begin{equation}
t^{\varepsilon }=\left\{ k\in \mathbb{Z}^{N}:\varepsilon \left( k+T\right)
\subset \Omega \right\} \text{,}  \label{eq1.1a}
\end{equation}%
\begin{equation}
T^{\varepsilon }=\underset{k\in t^{\varepsilon }}{\cup }\varepsilon \left(
k+T\right)  \label{eq1.1b}
\end{equation}%
and%
\begin{equation}
\Omega ^{\varepsilon }=\Omega \backslash T^{\varepsilon }\text{,}
\label{eq1.1c}
\end{equation}%
where $\mathbb{Z}$ denotes the integers. Throughout this study, $\Omega
^{\varepsilon }$ is a medium constituted of a porous membrane filled with
fluid, and $T$ is the reference solid part while $\varepsilon \left(
k+T\right) $ is an obstacle of size $\varepsilon $. The subset $%
T^{\varepsilon }$ is therefore the union of solid particules of size $%
\varepsilon $ in the porous domain $\Omega ^{\varepsilon }$. We denote by $%
\mathbf{n}=\left( n_{j}\right) _{1\leq j\leq N}$ the outward unit normal to $%
\partial T^{\varepsilon }$ with respect to $\Omega ^{\varepsilon }$.

For any Roman character such as $i$, $j$ (with $1\leq i,j\leq N$), $u^{i}$
(resp. $u^{j}$) denotes the $i$-th (resp. $j$-th) component of a vector
function $\mathbf{u}$ in $L_{loc}^{1}\left( \Omega \right) ^{N}$ or in $%
L_{loc}^{1}\left( \mathbb{R}^{N}\right) ^{N}$. Further, for any real $%
0<\varepsilon <1$, we define $u^{\varepsilon }$ as 
\begin{equation*}
u^{\varepsilon }\left( x\right) =u\left( \frac{x}{\varepsilon }\right) \text{%
\qquad }\left( x\in \Omega \right)
\end{equation*}%
for $u\in L_{loc}^{1}\left( \mathbb{R}_{y}^{N}\right) $. More generally, for 
$u\in L_{loc}^{1}\left( \Omega \times \mathbb{R}_{y}^{N}\right) $, it is
customary to put%
\begin{equation*}
u^{\varepsilon }\left( x\right) =u\left( x,\frac{x}{\varepsilon }\right) 
\text{\qquad }\left( x\in \Omega \right)
\end{equation*}%
whenever the right-hand side makes sense (see, e.g., \cite{bib9}).

Let $a_{ij}$ $\left( 1\leq i,j\leq N\right) $ and $\mathcal{\theta }$ be
real functions in $L^{\infty }\left( \mathbb{R}^{N}\right) $ such that: 
\begin{equation}
a_{ij}=a_{ji}\text{, and }\mathcal{\theta }\left( y\right) \geq \alpha _{0}%
\text{ a.e.in }y\in \mathbb{R}^{N}\text{,}  \label{eq1.2a}
\end{equation}%
\begin{equation}
\sum_{i,j=1}^{N}a_{ij}\left( y\right) \zeta _{j}\zeta _{i}\geq \alpha
\left\vert \zeta \right\vert ^{2}\text{\qquad }\left( \zeta =\left( \zeta
_{j}\right) \in \mathbb{R}^{N}\right) \text{ a.e. in }y\in \mathbb{R}^{N}%
\text{,}  \label{eq1.2}
\end{equation}%
where $\alpha >0$ and $\alpha _{0}>0$ are constants. We consider the partial
differential operator $P^{\varepsilon }$ on $\Omega $ defined by%
\begin{equation*}
P^{\varepsilon }=-\sum_{i,j=1}^{N}\frac{\partial }{\partial x_{i}}\left(
a_{ij}^{\varepsilon }\frac{\partial }{\partial x_{j}}\right) \text{.}
\end{equation*}%
The operator $P^{\varepsilon }$ acts on scalar functions, say $\varphi \in
H^{1}\left( \Omega \right) $. However, we may as well view $P^{\varepsilon }$
as acting on vector functions $\mathbf{u}=\left( u^{i}\right) \in
H^{1}\left( \Omega \right) ^{N}$ in a \textit{diagonal way}, i.e.,%
\begin{equation*}
\left( P^{\varepsilon }\mathbf{u}\right) ^{i}=P^{\varepsilon }u^{i}\text{%
\qquad }\left( i=1,...,N\right) \text{.}
\end{equation*}

For any fixed $0<\varepsilon <1$, we consider the boundary value problem%
\begin{equation}
P^{\varepsilon }\mathbf{u}_{\varepsilon }+\mathbf{grad}p_{\varepsilon }=%
\mathbf{f}^{\varepsilon }\text{ in }\Omega ^{\varepsilon }\text{,}
\label{eq1.3}
\end{equation}%
\begin{equation}
\qquad \qquad div\mathbf{u}_{\varepsilon }=0\text{ in }\Omega ^{\varepsilon }%
\text{,}  \label{eq1.4}
\end{equation}%
\begin{equation}
\qquad \qquad \qquad \mathbf{u}_{\varepsilon }=0\text{ on }\partial \Omega 
\text{,}  \label{eq1.5}
\end{equation}%
\begin{equation}
\text{ }-p_{\varepsilon }\mathbf{n}+\frac{\partial \mathbf{u}_{\varepsilon }%
}{\partial \mathbf{n}_{P^{\varepsilon }}}+\varepsilon \mathcal{\theta }%
^{\varepsilon }\mathbf{u}_{\varepsilon }=0\text{ on }\partial T^{\varepsilon
}\qquad \qquad  \label{eq1.7}
\end{equation}%
where $\mathbf{f}=\left( f_{j}\right) \in L^{\infty }\left( \mathbb{R}%
_{y}^{N}\right) ^{N}$ is a vector function with real components. For the
variational formulation of (\ref{eq1.3})-(\ref{eq1.7}), let us introduce%
\begin{equation*}
\mathbf{V}_{\varepsilon }=\left\{ \mathbf{v}\in H^{1}\left( \Omega
^{\varepsilon };\mathbb{R}\right) ^{N}:div\mathbf{v=}0\text{ in }\Omega
^{\varepsilon }\text{ and }\mathbf{v=}0\text{ on }\partial \Omega \right\} 
\text{,}
\end{equation*}%
where $H^{1}\left( \Omega ^{\varepsilon };\mathbb{R}\right) $ is the space
of\ functions in\ the Sobolev$\ $space$\ H^{1}\left( \Omega ^{\varepsilon
}\right) $ with real values, and let $\mathbf{a}_{\varepsilon }\left(
.,.\right) $ be the bilinear form on $H^{1}\left( \Omega ^{\varepsilon };%
\mathbb{R}\right) ^{N}$ given\ by%
\begin{equation*}
\mathbf{a}_{\varepsilon }\left( \mathbf{u},\mathbf{v}\right)
=\sum_{i,j,k=1}^{N}\int_{\Omega ^{\varepsilon }}a_{ij}^{\varepsilon }\frac{%
\partial u^{k}}{\partial x_{j}}\frac{\partial v^{k}}{\partial x_{i}}%
dx+\varepsilon \int_{\partial T^{\varepsilon }}\mathcal{\theta }%
^{\varepsilon }\mathbf{u}{\small \cdot }\mathbf{v}d\sigma _{\varepsilon }
\end{equation*}%
for $\mathbf{u=}\left( u^{k}\right) $ and $\mathbf{v=}\left( v^{k}\right)
\in H^{1}\left( \Omega ^{\varepsilon };\mathbb{R}\right) ^{N}$, the dot
denoting the Euclidean inner product, and $d\sigma _{\varepsilon }$ being
the surface measure on $\partial T^{\varepsilon }$. The boundary value
problem\ (\ref{eq1.3})-(\ref{eq1.7}) naturally\ implies the following
variational equation:%
\begin{equation}
\left\{ 
\begin{array}{c}
\mathbf{u}_{\varepsilon }\in \mathbf{V}_{\varepsilon },\qquad \qquad \qquad
\qquad \qquad \quad \\ 
\mathbf{a}_{\varepsilon }\left( \mathbf{u}_{\varepsilon },\mathbf{v}\right)
=\int_{\Omega ^{\varepsilon }}\mathbf{f}^{\varepsilon }{\small \cdot }%
\mathbf{v}dx\text{ for all }\mathbf{v}\in \mathbf{V}_{\varepsilon }\text{.}%
\end{array}%
\right.  \label{eq1.8}
\end{equation}%
The variational\ problem\ (\ref{eq1.8}) is a classical\ one which\ admits a
unique solution, in view of\ (\ref{eq1.2a})-(\ref{eq1.2}). Further, it is
easy\ to check that (\ref{eq1.8}) leads to\ (\ref{eq1.3})-(\ref{eq1.7}).
Thus, the problem (\ref{eq1.3})-(\ref{eq1.7}) admits a unique solution $%
\left( \mathbf{u}_{\varepsilon },p_{\varepsilon }\right) $\ in $\mathbf{V}%
_{\varepsilon }\times \left( L^{2}\left( \Omega ^{\varepsilon };\mathbb{R}%
\right) \mathfrak{/}\mathbb{R}\right) $, where 
\begin{equation*}
L^{2}\left( \Omega ^{\varepsilon };\mathbb{R}\right) \mathfrak{/}\mathbb{R=}%
\left\{ v\in L^{2}\left( \Omega ^{\varepsilon };\mathbb{R}\right)
:\int_{\Omega ^{\varepsilon }}v\left( x\right) dx=0\right\} \text{.}
\end{equation*}

Our aim here is to investigate the asymptotic behavior, as $\varepsilon
\rightarrow 0$, of $\left( \mathbf{u}_{\varepsilon },p_{\varepsilon }\right) 
$ under the hypotheses that 
\begin{equation}
a_{ij}\left( y+k\right) =a_{ij}\left( y\right) \text{,\quad }\mathcal{\theta 
}\left( y+k\right) =\mathcal{\theta }\left( y\right) \text{ and }\mathbf{f}%
\left( y+k\right) =\mathbf{f}\left( y\right) \quad \left( 1\leq i,j\leq
N\right)  \label{eq1.8a}
\end{equation}%
for almost all $y\in \mathbb{R}^{N}$ and for all $k\in \mathbb{Z}^{N}$.

The study of this problem turns out to be of benefit to the modelling of
heterogeneous fluid flows, in particular multi-phase flows, fluids with
spatially varying viscosities, and others. This model has been considered
for the first time in \cite{bib3} in the case of a fixe domain $\Omega $. In 
\cite{bib6}, Conca has studied the homogenization of a classical
two-dimensional Stokes flow with the Fourier-Neumann conditions on the
boundary of the holes of a periodically perforated domain. We mention also
the paper by Cioranescu, Donato and Ene \cite{bib4} dealing with that
classical Stokes problem with exterior surface forces induced by an
electrical field on the boundary of the holes. In \cite{bib7} the
homogenization of the steady Stokes equations in a domain containing a
periodically perforated sieve with a non-slip boundary condition on the
sieve has also been investigated.

This paper deals with a case where the usual laplace operator is replace by
an elliptic linear partial differential operator of order two with
oscillating coefficients. As in \cite{bib6}, this paper is concerned with
the study of the asymtotic behavior of an incompressible fluid flow in
porous media with a purely mechanical slip boundary conditions.

By means of the \textit{two-scale convergence }techniques, which are nothing
but the \textit{sigma-convergence }in the periodic setting,\textit{\ }we
derive the homogenized model for (\ref{eq1.3})-(\ref{eq1.7}).

Unless otherwise specified, vector spaces throughout are considered over the
complex field, $\mathbb{C}$, and scalar functions are assumed to take
complex values. Let us recall some basic notations. If $X$ and $F$ denote a
locally compact space and a Banach space, respectively, then we write $%
\mathcal{C}\left( X;F\right) $ for continuous mappings of $X$ into $F$, and $%
\mathcal{B}\left( X;F\right) $ for those mappings in $\mathcal{C}\left(
X;F\right) $ that are bounded. We denote by $\mathcal{K}\left( X;F\right) $
the mappings in $\mathcal{C}\left( X;F\right) $ having compact supports.\ We
shall assume $\mathcal{B}\left( X;F\right) $ to be equipped with the
supremum norm $\left\Vert u\right\Vert _{\infty }=\sup_{x\in X}\left\Vert
u\left( x\right) \right\Vert $ ($\left\Vert {\small \cdot }\right\Vert $
denotes the norm in $F$). For shortness we will write $\mathcal{C}\left(
X\right) =\mathcal{C}\left( X;\mathbb{C}\right) $, $\mathcal{B}\left(
X\right) =\mathcal{B}\left( X;\mathbb{C}\right) $ and $\mathcal{K}\left(
X\right) =\mathcal{K}\left( X;\mathbb{C}\right) $. Likewise in the case when 
$F=\mathbb{C}$, the usual spaces $L^{p}\left( X;F\right) $ and $%
L_{loc}^{p}\left( X;F\right) $ ($X$ provided with a positive Radon measure)
will be denoted by $L^{p}\left( X\right) $ and $L_{loc}^{p}\left( X\right) $%
, respectively. Finally, the numerical space $\mathbb{R}^{N}$ and its open
sets are each provided with Lebesgue measure denoted by $dx=dx_{1}...dx_{N}$.

The rest of the paper is organized as follows. Section 2 is devoted to the
preliminaries while in Section 3, a convergence result is proved for (\ref%
{eq1.3})-(\ref{eq1.7}).

\section{Preliminary results}

Before we begin with preliminaries, let us note that, if $\mathbf{w=}\left(
w^{k}\right) _{1\leq k\leq N}$ with $w^{k}\in L^{p}\left( \mathcal{O}\right) 
$, or if $\mathbf{w=}\left( w^{ij}\right) _{1\leq i,j\leq N}$ with $%
w^{ij}\in L^{p}\left( \mathcal{O}\right) $, where $\mathcal{O}$ is an open
set in $\mathbb{R}^{N}$, we will sometimes write $\left\Vert \mathbf{w}%
\right\Vert _{L^{p}\left( \mathcal{O}\right) }$ for $\left\Vert \mathbf{w}%
\right\Vert _{L^{p}\left( \mathcal{O}\right) ^{N}}$ or for $\left\Vert 
\mathbf{w}\right\Vert _{L^{p}\left( \mathcal{O}\right) ^{N\times N}}$.

Let us first recall the following result on the construction (for $%
\varepsilon >0$) of a suitable extension operator sending $\mathbf{V}%
_{\varepsilon }$ into $H_{0}^{1}\left( \Omega ;\mathbb{R}\right) ^{N}$.

\begin{proposition}
\label{pr2.1} For each real $\varepsilon >0$, there exists an operator $%
\mathcal{P}_{\varepsilon }$ of $\mathbf{V}_{\varepsilon }$ into $%
H_{0}^{1}\left( \Omega ;\mathbb{R}\right) ^{N}$ with the following
properties: 
\begin{equation}
\mathcal{P}_{\varepsilon }\text{ sends continuously and linearly }\mathbf{V}%
_{\varepsilon }\text{ into }H_{0}^{1}\left( \Omega ;\mathbb{R}\right) ^{N}%
\text{;}  \label{eq2.1}
\end{equation}%
\begin{equation}
\left( \mathcal{P}_{\varepsilon }\mathbf{v}\right) {\small \mid }_{\Omega
^{\varepsilon }}=\mathbf{v}\text{ for all }\mathbf{v}\in \mathbf{V}%
_{\varepsilon }\text{;}  \label{eq2.2}
\end{equation}%
\begin{equation}
\left\Vert \nabla \left( \mathcal{P}_{\varepsilon }\mathbf{v}\right)
\right\Vert _{L^{2}\left( \Omega \right) }\leq c\left\Vert \nabla \mathbf{v}%
\right\Vert _{L^{2}\left( \Omega ^{\varepsilon }\right) }\text{ for all }%
\mathbf{v}\in \mathbf{V}_{\varepsilon }  \label{eq2.3}
\end{equation}%
where the constant $c>0$ depends solely on $Y$ and $T$.
\end{proposition}

\begin{proof}
The proof of this Proposition follows the same line of argument as in \cite[%
proof of Lemma 7.1]{bib6} (see also \cite{bib5}). The details are left to
the reader.
\end{proof}

Let us go to our next purpose. We set 
\begin{equation}
\Theta =\tbigcup_{k\in \mathbb{Z}^{N}}\left( k+T\right) \text{.}
\label{eq2.6a}
\end{equation}%
Using the compactness of $T$, it is easy to check that $\Theta $ is closed
in $\mathbb{R}^{N}$. Let us set for any $\varepsilon >0$ 
\begin{equation}
Q^{\varepsilon }=\Omega \backslash \varepsilon \Theta \text{.}
\label{eq2.6b}
\end{equation}%
Then $Q^{\varepsilon }$ is an open set of $\mathbb{R}^{N}$ and clearly, $%
Q^{\varepsilon }\subset \Omega ^{\varepsilon }$. The set $Q^{\varepsilon }$
is made of two types of solid particules: on one hand, the solids of $\Omega
^{\varepsilon }$, on the other hand, the solids $\Omega \cap \varepsilon
\left( k+T\right) $ where $\varepsilon \left( k+T\right) $ intersects $%
\partial \Omega $. In view of the following results we well see that the
difference $\Omega ^{\varepsilon }\backslash Q^{\varepsilon }$ is of no
effect in the homogenization process.

\begin{lemma}
\label{lem2.2} Let $K\subset \Omega $ be a compact set ($K$ independent of $%
\varepsilon $). There exists some $\varepsilon _{0}>0$ such that $\Omega
^{\varepsilon }\backslash Q^{\varepsilon }\subset \left( \Omega \backslash
K\right) $ for any $0<\varepsilon \leq \varepsilon _{0}$.
\end{lemma}

\begin{proof}
Let $U$ be a neighborhood of $\partial \Omega $ in $\mathbb{R}^{N}$ such
that $U\subset \left( \mathbb{R}^{N}\backslash K\right) $. The first point
is to show that a real $\varepsilon _{0}>0$ exists such that 
\begin{equation}
J^{\varepsilon }\left( \partial \Omega \right) \subset U\text{ for any }%
0<\varepsilon \leq \varepsilon _{0}  \label{eq2.7}
\end{equation}%
where $J^{\varepsilon }\left( \partial \Omega \right) $ denotes the union of
all $\varepsilon \left( k+\overline{Y}\right) $ as $k$ ranges over $%
j^{\varepsilon }\left( \partial \Omega \right) =\left\{ k\in \mathbb{Z}%
^{N}:\varepsilon \left( k+\overline{Y}\right) \cap \partial \Omega \neq
\varnothing \right\} $. But, as the sets 
\begin{equation*}
U_{r}=\left\{ x\in \mathbb{R}^{N}:d\left( x,\partial \Omega \right) \leq
r\right\}
\end{equation*}%
(where $d$ denotes the Euclidean metric in $\mathbb{R}^{N}$ and $r$ ranges
over positive real numbers) form the base of neighborhoods of $\partial
\Omega $, we see that we may assume without loss of generality that $U=U_{r}$
for a suitable $r>0$. Then, according to \cite[Lemma 1]{bib8}, there is some 
$\varepsilon _{0}>0$ such that $J^{\varepsilon }\left( \partial \Omega
\right) \subset U_{r}$ for any $0<\varepsilon \leq \varepsilon _{0}$, which
shows (\ref{eq2.7}). The next point is to show that 
\begin{equation}
\Omega ^{\varepsilon }\backslash Q^{\varepsilon }\subset J^{\varepsilon
}\left( \partial \Omega \right) \text{\qquad }\left( \varepsilon >0\right) 
\text{.}  \label{eq2.8}
\end{equation}%
Let $x\in \Omega ^{\varepsilon }\backslash Q^{\varepsilon }$, then $x\in
\Omega ^{\varepsilon }$and $x\in \varepsilon \Theta $. Thus, there exists
some $k\in \mathbb{Z}^{N}$ such that $k\notin t^{\varepsilon }$ and $x\in
\varepsilon \left( k+T\right) $. Hence, $\varepsilon \left( k+\overline{Y}%
\right) $ is not contained in $\Omega $ and $\varepsilon \left( k+\overline{Y%
}\right) \cap \Omega \neq \varnothing $. Since $\overline{Y}$ is connected,
we have $\varepsilon \left( k+\overline{Y}\right) \cap \partial \Omega \neq
\varnothing $ and thereby $k\in j^{\varepsilon }\left( \partial \Omega
\right) $. Thus, $x\in J^{\varepsilon }\left( \partial \Omega \right) $ and (%
\ref{eq2.8}) holds true. The lemma follows at once by (\ref{eq2.7}) and (\ref%
{eq2.8}).
\end{proof}

\begin{remark}
\label{rem2.1} As $\varepsilon \rightarrow 0$ we have $\lambda \left( \Omega
^{\varepsilon }\backslash Q^{\varepsilon }\right) \rightarrow 0$ ($\lambda $
denotes Lebesgue measure on $\mathbb{R}^{N}$). Indeed, this follows by
combining (\ref{eq2.8}) with the fact that $\lambda \left( J^{\varepsilon
}\left( \partial \Omega \right) \right) \rightarrow \lambda \left( \partial
\Omega \right) =0$ (see \cite[proof of Theorem 4.3]{bib11}).
\end{remark}

Now, let us turn to some fundamental preliminary results on the \textit{%
sigma-convergence }in the periodic setting.

Let us first recall that a function $u\in L_{loc}^{1}\left( \mathbb{R}%
_{y}^{N}\right) $ is said to be $Y$-periodic if for each $k\in \mathbb{Z}%
^{N} $, we have $u\left( y+k\right) =u\left( y\right) $ almost everywhere
(a.e.) in $y\in \mathbb{R}^{N}$. If in addition $u$ is continuous, then the
preceding equality holds for every $y\in \mathbb{R}^{N}$, of course. The
space of all $Y$-periodic continuous complex functions on $\mathbb{R}%
_{y}^{N} $ is denoted by $\mathcal{C}_{per}\left( Y\right) $; that of all $Y$%
-periodic functions in $L_{loc}^{p}\left( \mathbb{R}_{y}^{N}\right) $ $%
\left( 1\leq p<\infty \right) $ is denoted by $L_{per}^{p}\left( Y\right) $. 
$\mathcal{C}_{per}\left( Y\right) $ is a Banach space under the supremum
norm on $\mathbb{R}^{N}$, whereas $L_{per}^{p}\left( Y\right) $ is a Banach
space under the norm 
\begin{equation*}
\left\Vert u\right\Vert _{L^{p}\left( Y\right) }=\left( \int_{Y}\left\vert
u\left( y\right) \right\vert ^{p}dy\right) ^{\frac{1}{p}}\text{ }\left( u\in
L_{per}^{p}\left( Y\right) \right) \text{.}
\end{equation*}

We will need the space $H_{per}^{1}\left( Y\right) $ of functions in $%
H_{loc}^{1}\left( \mathbb{R}_{y}^{N}\right) =W_{loc}^{1,2}\left( \mathbb{R}%
_{y}^{N}\right) $ which are $Y$-periodic, and the space $H_{\#}^{1}\left(
Y\right) $ of functions $u\in H_{per}^{1}\left( Y\right) $ such that $%
\int_{Y}\left( y\right) dy=0$. Provided with the gradient norm, 
\begin{equation*}
\left\Vert u\right\Vert _{H_{\#}^{1}\left( Y\right) }=\left(
\int_{Y}\left\vert \nabla _{y}u\right\vert ^{2}dy\right) ^{\frac{1}{2}}\text{
}\left( u\in H_{\#}^{1}\left( Y\right) \right) \text{,}
\end{equation*}%
where $\nabla _{y}u=\left( \frac{\partial u}{\partial y_{1}},...,\frac{%
\partial u}{\partial y_{N}}\right) $, $H_{\#}^{1}\left( Y\right) $ is a
Hilbert space.

Before we can recall the concept of \textit{sigma-convergence} in the
present periodic setting or the \textit{two-scale convergence}, let us
introduce one further notation. The letter $E$ throughout will denote a
family of real numbers $0<\varepsilon <1$ admitting $0$ as an accumulation
point. For example, $E$ may be the whole interval $\left( 0,1\right) $; $E$
may also be an ordinary sequence $\left( \varepsilon _{n}\right) _{n\in 
\mathbb{N}}$ with $0<\varepsilon _{n}<1$ and $\varepsilon _{n}\rightarrow 0$
as $n\rightarrow \infty $. In the latter case $E$ will be referred to as a 
\textit{fundamental sequence}.

Let $\Omega $ be a bounded open set in $\mathbb{R}_{x}^{N}$ and let $1\leq
p<\infty $.

\begin{definition}
\label{def2.1} A sequence $\left( u_{\varepsilon }\right) _{\varepsilon \in
E}\subset L^{p}\left( \Omega \right) $ is said to:

(i) weakly $\Sigma $-converge in $L^{p}\left( \Omega \right) $ to some $%
u_{0}\in L^{p}\left( \Omega ;L_{per}^{p}\left( Y\right) \right) $ if as

\noindent $E\ni \varepsilon \rightarrow 0$, 
\begin{equation}
\int_{\Omega }u_{\varepsilon }\left( x\right) \psi ^{\varepsilon }\left(
x\right) dx\rightarrow \int \int_{\Omega \times Y}u_{0}\left( x,y\right)
\psi \left( x,y\right) dxdy  \label{eq2.9}
\end{equation}%
\begin{equation*}
\begin{array}{c}
\text{for all }\psi \in L^{p^{\prime }}\left( \Omega ;\mathcal{C}%
_{per}\left( Y\right) \right) \text{ }\left( \frac{1}{p^{\prime }}=1-\frac{1%
}{p}\right) \text{, where }\psi ^{\varepsilon }\left( x\right) = \\ 
\psi \left( x,\frac{x}{\varepsilon }\right) \text{ }\left( x\in \Omega
\right) \text{;}%
\end{array}%
\end{equation*}

(ii) strongly $\Sigma $-converge in $L^{p}\left( \Omega \right) $ to some $%
u_{0}\in L^{p}\left( \Omega ;L_{per}^{p}\left( Y\right) \right) $ if the
following property is verified: 
\begin{equation*}
\left\{ 
\begin{array}{c}
\text{Given }\eta >0\text{ and }v\in L^{p}\left( \Omega ;\mathcal{C}%
_{per}\left( Y\right) \right) \text{ with} \\ 
\left\Vert u_{0}-v\right\Vert _{L^{p}\left( \Omega \times Y\right) }\leq 
\frac{\eta }{2}\text{, there is some }\alpha >0\text{ such} \\ 
\text{that }\left\Vert u_{\varepsilon }-v^{\varepsilon }\right\Vert
_{L^{p}\left( \Omega \right) }\leq \eta \text{ provided }E\ni \varepsilon
\leq \alpha \text{.}%
\end{array}%
\right.
\end{equation*}
\end{definition}

We will briefly express weak and strong $\Sigma $-convergence by writing $%
u_{\varepsilon }\rightarrow u_{0}$ in $L^{p}\left( \Omega \right) $-weak $%
\Sigma $ and $u_{\varepsilon }\rightarrow u_{0}$ in $L^{p}\left( \Omega
\right) $-strong $\Sigma $, respectively. Instead of repeating here the main
results underlying $\Sigma $-convergence theory for periodic structures, we
find it more convenient to draw the reader's attention to a few references
regarding two-scale convergence, e.g., \cite{bib1}, \cite{bib2}, \cite{bib8}
and \cite{bib9}. However, we recall below two fundamental results.

\begin{theorem}
\label{th2.1} Assume that $1<p<\infty $ and further $E$ is a fundamental
sequence. Let a sequence $\left( u_{\varepsilon }\right) _{\varepsilon \in
E} $ be bounded in $L^{p}\left( \Omega \right) $. Then, a subsequence $%
E^{\prime }$ can be extracted from $E$ such that $\left( u_{\varepsilon
}\right) _{\varepsilon \in E^{\prime }}$ weakly $\Sigma $-converges in $%
L^{p}\left( \Omega \right) $.
\end{theorem}

\begin{theorem}
\label{th2.2} Let $E$ be a fundamental sequence. Suppose a sequence $\left(
u_{\varepsilon }\right) _{\varepsilon \in E}$ is bounded in $H^{1}\left(
\Omega \right) =W^{1,2}\left( \Omega \right) $. Then, a subsequence $%
E^{\prime }$ can be extracted from $E$ such that, as $E^{\prime }\ni
\varepsilon \rightarrow 0$, 
\begin{equation*}
u_{\varepsilon }\rightarrow u_{0}\text{ in }H^{1}\left( \Omega \right) \text{%
-weak,\qquad \qquad \qquad \qquad \qquad \qquad }
\end{equation*}%
\begin{equation*}
u_{\varepsilon }\rightarrow u_{0}\text{ in }L^{2}\left( \Omega \right) \text{%
-weak }\Sigma \text{,\qquad \qquad \qquad \qquad \qquad }\quad
\end{equation*}%
\begin{equation*}
\frac{\partial u_{\varepsilon }}{\partial x_{j}}\rightarrow \frac{\partial
u_{0}}{\partial x_{j}}+\frac{\partial u_{1}}{\partial y_{j}}\text{ in }%
L^{2}\left( \Omega \right) \text{-weak }\Sigma \text{ }\left( 1\leq j\leq
N\right) \text{,}
\end{equation*}%
where $u_{0}\in H^{1}\left( \Omega \right) $, $u_{1}\in L^{2}\left( \Omega
;H_{\#}^{1}\left( Y\right) \right) $.
\end{theorem}

Now, let us also introduce the notion of \textit{two-scale convergence on
periodic surfaces. We denote by }$L_{per}^{p}\left( \partial T\right) $ the
space of functions $u$ in $L_{loc}^{p}\left( \partial \Theta \right) $
verifying $u\left( y+k\right) =u\left( y\right) $ for all $k\in \mathbb{Z}%
^{N}$ and for almost all $y\in \partial \Theta $ ($\partial \Theta $ is the
boundary of $\Theta $). Let $\partial T^{\varepsilon }$ be the boundary of $%
T^{\varepsilon }$ ($T^{\varepsilon }$ is given by (\ref{eq1.1b})).

\begin{definition}
\label{def2.2} A sequence $\left( u_{\varepsilon }\right) _{\varepsilon \in
E}$ with $u_{\varepsilon }\in L^{p}\left( \partial T^{\varepsilon }\right) $
for all $\varepsilon \in E$ is said to two-scale converge to some $u_{0}\in
L^{p}\left( \Omega ;L_{per}^{p}\left( \partial T\right) \right) $ if as $%
E\ni \varepsilon \rightarrow 0$,%
\begin{equation}
\varepsilon \int_{\partial T^{\varepsilon }}u_{\varepsilon }\left( x\right)
\psi ^{\varepsilon }\left( x\right) d\sigma _{\varepsilon }\left( x\right)
\rightarrow \int \int_{\Omega \times \partial T}u_{0}\left( x,y\right) \psi
\left( x,y\right) dxd\sigma \left( y\right)  \label{eq2.9a}
\end{equation}%
\begin{equation*}
\begin{array}{c}
\text{for all }\psi \in \mathcal{C}\left( \overline{\Omega };\mathcal{C}%
_{per}\left( Y\right) \right) \text{, where }\psi ^{\varepsilon }\left(
x\right) = \\ 
\psi \left( x,\frac{x}{\varepsilon }\right) \text{ }\left( x\in \Omega
\right)%
\end{array}%
\end{equation*}%
and where $d\sigma _{\varepsilon }$ and $d\sigma $ denote the surface
measures on $\partial T^{\varepsilon }$ and $\partial T$, respectively.
\end{definition}

The following result of convergence on periodic surfaces holds true.

\begin{theorem}
\label{th2.3} Let $1<p<+\infty $, and let $\left( u_{\varepsilon }\right)
_{\varepsilon \in E}$ be a sequence with $u_{\varepsilon }\in L^{p}\left(
\partial T^{\varepsilon }\right) $ for all $\varepsilon \in E$. Suppose that 
\begin{equation}
\varepsilon \int_{\partial T^{\varepsilon }}\left\vert u_{\varepsilon
}\left( x\right) \right\vert ^{p}d\sigma _{\varepsilon }\left( x\right) \leq
C  \label{eq2.9b}
\end{equation}
for all $\varepsilon \in E$, where $C$ is a constant independent of $%
\varepsilon $. Then, there exists a subsequence $E^{\prime }$ extracted from 
$E$ and a function $u_{0}\in L^{p}\left( \Omega ;L_{per}^{p}\left( \partial
T\right) \right) $ such that $\left( u_{\varepsilon }\right) _{\varepsilon
\in E^{\prime }}$ two-scale converges to $u_{0}$.
\end{theorem}

The proof of the preceding theorem is based on the following lemma which is
easy to establish.

\begin{lemma}
\label{lem2.3} Let $\psi \in \mathcal{C}\left( \overline{\Omega };\mathcal{C}%
_{per}\left( Y\right) \right) $ and $1<p<+\infty $. There exists a constant $%
C_{0}$ independent of $\varepsilon $ such that 
\begin{equation}
\varepsilon \int_{\partial T^{\varepsilon }}\left\vert \psi \left( x,\frac{x%
}{\varepsilon }\right) \right\vert ^{p}d\sigma _{\varepsilon }\left(
x\right) \leq C_{0}\left\Vert \psi \right\Vert _{\infty }^{p}\text{ }
\label{eq2.9c}
\end{equation}%
for all $\varepsilon >0$. Moreover, as $\varepsilon \rightarrow 0$, 
\begin{equation}
\varepsilon \int_{\partial T^{\varepsilon }}\left\vert \psi \left( x,\frac{x%
}{\varepsilon }\right) \right\vert ^{p}d\sigma _{\varepsilon }\left(
x\right) \rightarrow \int \int_{\Omega \times \partial T}\left\vert \psi
\left( x,y\right) \right\vert ^{p}dxd\sigma \left( y\right) \text{.}
\label{eq2.9d}
\end{equation}
\end{lemma}

\begin{proof}[Proof of Theorem \protect\ref{th2.3}]
For all $\psi \in \mathcal{C}\left( \overline{\Omega };\mathcal{C}%
_{per}\left( Y\right) \right) $ we set 
\begin{equation*}
F_{\varepsilon }\left( \psi \right) =\varepsilon \int_{\partial
T^{\varepsilon }}u_{\varepsilon }\left( x\right) \psi \left( x,\frac{x}{%
\varepsilon }\right) d\sigma _{\varepsilon }\left( x\right) \text{.}
\end{equation*}%
This defines a linear form on $\mathcal{C}\left( \overline{\Omega };\mathcal{%
C}_{per}\left( Y\right) \right) $. In view of (\ref{eq2.9b}) and (\ref%
{eq2.9c}), one has%
\begin{equation}
\left\vert F_{\varepsilon }\left( \psi \right) \right\vert \leq C^{\frac{1}{p%
}}\left( \varepsilon \int_{\partial T^{\varepsilon }}\left\vert \psi \left(
x,\frac{x}{\varepsilon }\right) \right\vert ^{p^{\prime }}d\sigma
_{\varepsilon }\left( x\right) \right) ^{\frac{1}{p^{\prime }}}\leq C^{\frac{%
1}{p}}C_{0}^{\frac{1}{p^{\prime }}}\left\Vert \psi \right\Vert _{\infty }
\label{eq2.10}
\end{equation}%
for all $\psi \in \mathcal{C}\left( \overline{\Omega };\mathcal{C}%
_{per}\left( Y\right) \right) $, where $p^{\prime }=\frac{p}{p-1}$. Thus $%
F_{\varepsilon }$ is a continuous linear form on $\mathcal{C}\left( 
\overline{\Omega };\mathcal{C}_{per}\left( Y\right) \right) $ for all $%
\varepsilon >0$, and further $\left( F_{\varepsilon }\right) _{\in E}$ is a
bounded sequence in the topological dual of the separable Banach space $%
\mathcal{C}\left( \overline{\Omega };\mathcal{C}_{per}\left( Y\right)
\right) $. Therefore, the Banach-Alaoglu theorem yields the existence of a
subsequence $E^{\prime }$ extracted from $E$ and a continuous linear form $%
F_{0}$ on $\mathcal{C}\left( \overline{\Omega };\mathcal{C}_{per}\left(
Y\right) \right) $ such that, as $E^{\prime }\ni \varepsilon \rightarrow 0$, 
\begin{equation*}
F_{\varepsilon }\left( \psi \right) \rightarrow F_{0}\left( \psi \right)
\end{equation*}%
for all $\psi \in \mathcal{C}\left( \overline{\Omega };\mathcal{C}%
_{per}\left( Y\right) \right) $. Using (\ref{eq2.9d}) and passing to the
limit in (\ref{eq2.10}) as $E^{\prime }\ni \varepsilon \rightarrow 0$ lead to%
\begin{equation*}
\left\vert F_{0}\left( \psi \right) \right\vert \leq C^{\frac{1}{p}}\left(
\int \int_{\Omega \times \partial T}\left\vert \psi \left( x,y\right)
\right\vert ^{p^{\prime }}dxd\sigma \left( y\right) \right) ^{\frac{1}{%
p^{\prime }}}=C^{\frac{1}{p}}\left\Vert \psi \right\Vert _{L^{p^{\prime
}}\left( \Omega ;L_{per}^{p^{\prime }}\left( \partial T\right) \right) }%
\text{,}
\end{equation*}%
for all $\psi \in \mathcal{C}\left( \overline{\Omega };\mathcal{C}%
_{per}\left( Y\right) \right) $. Consequently, $F_{0}$ can be extended to a
unique continuous linear form on $L^{p^{\prime }}\left( \Omega
;L_{per}^{p^{\prime }}\left( \partial T\right) \right) $, since the
restrictions on $\Omega \times \partial T$ of functions in $\mathcal{C}%
\left( \overline{\Omega };\mathcal{C}_{per}\left( Y\right) \right) $ are
dense in $L^{p^{\prime }}\left( \Omega ;L_{per}^{p^{\prime }}\left( \partial
T\right) \right) $. Thus, by the Riesz representation theorem, there exists
a unique $u_{0}\in L^{p}\left( \Omega ;L_{per}^{p}\left( \partial T\right)
\right) $ such that 
\begin{equation*}
F_{0}\left( \psi \right) =\int \int_{\Omega \times \partial T}u_{0}\left(
x,y\right) \psi \left( x,y\right) dxd\sigma \left( y\right) \text{,}
\end{equation*}%
for all $\psi \in L^{p^{\prime }}\left( \Omega ;L_{per}^{p^{\prime }}\left(
\partial T\right) \right) $, and we have (\ref{eq2.9a}) for all $\psi \in 
\mathcal{C}\left( \overline{\Omega };\mathcal{C}_{per}\left( Y\right)
\right) $.
\end{proof}

\begin{remark}
\label{rem2.2} It is of interest to know that if $u_{\varepsilon
}\rightarrow u_{0}$ in $L^{p}\left( \Omega \right) $-weak $\Sigma $, then (%
\ref{eq2.9}) holds for $\psi \in \mathcal{C}\left( \overline{\Omega }%
;L_{per}^{\infty }\left( Y\right) \right) $ (see \cite[Proposition 10]{bib10}
for the proof). Moreover if $\left( u_{\varepsilon }\right) _{\varepsilon
>0} $ with $u_{\varepsilon }\in L^{p}\left( \partial T^{\varepsilon }\right) 
$ two-scale converges to $u_{0}\in L^{p}\left( \Omega ;L_{per}^{p}\left(
\partial T\right) \right) $ (in the sense of Definition \ref{def2.2}), then (%
\ref{eq2.9a}) holds for $\psi \in \mathcal{C}\left( \overline{\Omega }%
;L_{per}^{\infty }\left( Y\right) \right) $. Indeed, for $\psi \in \mathcal{C%
}\left( \overline{\Omega };L_{per}^{\infty }\left( Y\right) \right) $ and $%
\eta >0$, there exists (by density) some $w\in \mathcal{C}\left( \overline{%
\Omega };\mathcal{C}_{per}\left( Y\right) \right) $ such that 
\begin{equation*}
\left\Vert \psi -w\right\Vert _{\infty }\leq \frac{\eta }{3c}\text{,}
\end{equation*}%
where 
\begin{equation*}
c\geq \sup \left\{ \left\vert \Omega \right\vert ^{\frac{1}{p^{\prime }}%
}\left\vert \partial T\right\vert ^{\frac{1}{p^{\prime }}}\left\Vert
u_{0}\right\Vert _{L^{p}\left( \Omega ;L_{per}^{p}\left( \partial T\right)
\right) }\text{, }C^{\frac{1}{p}}C_{0}^{\frac{1}{p^{\prime }}}\right\}
\end{equation*}%
and where $C$ and $C_{0}$ are the same constants in (\ref{eq2.10}). Further,
we have 
\begin{equation*}
\begin{array}{c}
\varepsilon \int_{\partial T^{\varepsilon }}u_{\varepsilon }\left( x\right)
\psi \left( x,\frac{x}{\varepsilon }\right) d\sigma _{\varepsilon }\left(
x\right) -\int \int_{\Omega \times \partial T}u_{0}\left( x,y\right) \psi
\left( x,y\right) dxd\sigma \left( y\right) = \\ 
\varepsilon \int_{\partial T^{\varepsilon }}u_{\varepsilon }\left( x\right)
\left( \psi \left( x,\frac{x}{\varepsilon }\right) -w\left( x,\frac{x}{%
\varepsilon }\right) \right) d\sigma _{\varepsilon }\left( x\right) +\int
\int_{\Omega \times \partial T}u_{0}\left( x,y\right) \left( w\left(
x,y\right) -\psi \left( x,y\right) \right) dxd\sigma \left( y\right) + \\ 
\varepsilon \int_{\partial T^{\varepsilon }}u_{\varepsilon }\left( x\right)
w\left( x,\frac{x}{\varepsilon }\right) d\sigma _{\varepsilon }\left(
x\right) -\int \int_{\Omega \times \partial T}u_{0}\left( x,y\right) w\left(
x,y\right) dxd\sigma \left( y\right) \text{.}%
\end{array}%
\end{equation*}%
Since (\ref{eq2.9a}) holds for $w$, there exists some $\alpha >0$ such $E\ni
\varepsilon \leq \alpha $ implies%
\begin{equation*}
\left\vert \varepsilon \int_{\partial T^{\varepsilon }}u_{\varepsilon
}\left( x\right) w\left( x,\frac{x}{\varepsilon }\right) d\sigma
_{\varepsilon }\left( x\right) -\int \int_{\Omega \times \partial
T}u_{0}\left( x,y\right) w\left( x,y\right) dxd\sigma \left( y\right)
\right\vert \leq \frac{\eta }{3}\text{. }
\end{equation*}%
Moreover, 
\begin{equation*}
\left\vert \varepsilon \int_{\partial T^{\varepsilon }}u_{\varepsilon
}\left( x\right) \left( \psi \left( x,\frac{x}{\varepsilon }\right) -w\left(
x,\frac{x}{\varepsilon }\right) \right) d\sigma _{\varepsilon }\left(
x\right) \right\vert \leq C^{\frac{1}{p}}C_{0}^{\frac{1}{p^{\prime }}%
}\left\Vert \psi -w\right\Vert _{\infty }\leq \frac{\eta }{3}
\end{equation*}%
and 
\begin{equation*}
\left\vert \int \int_{\Omega \times \partial T}u_{0}\left( x,y\right) \left(
w\left( x,y\right) -\psi \left( x,y\right) \right) dxd\sigma \left( y\right)
\right\vert \leq \left\vert \Omega \right\vert ^{\frac{1}{p^{\prime }}%
}\left\vert \partial T\right\vert ^{\frac{1}{p^{\prime }}}\left\Vert
u_{0}\right\Vert _{L^{p}\left( \Omega ;L_{per}^{p}\left( \partial T\right)
\right) }\left\Vert \psi -w\right\Vert _{\infty }
\end{equation*}%
\begin{equation*}
\leq \frac{\eta }{3}\text{.}
\end{equation*}%
Thus $E\ni \varepsilon \leq \alpha $ implies 
\begin{equation*}
\left\vert \varepsilon \int_{\partial T^{\varepsilon }}u_{\varepsilon
}\left( x\right) \psi \left( x,\frac{x}{\varepsilon }\right) d\sigma
_{\varepsilon }\left( x\right) -\int \int_{\Omega \times \partial
T}u_{0}\left( x,y\right) \psi \left( x,y\right) dxd\sigma \left( y\right)
\right\vert \leq \eta \text{.}
\end{equation*}
\end{remark}

Let us state now the following useful proposition which is proved in \cite%
{bib2} .

\begin{proposition}
\label{pr2.2} Let $\left( u_{\varepsilon }\right) _{\varepsilon \in E}$ be a
sequence in $H^{1}\left( \Omega \right) $ such that 
\begin{equation*}
\left\Vert u_{\varepsilon }\right\Vert _{L^{2}\left( \Omega \right)
}+\varepsilon \left\Vert \nabla u_{\varepsilon }\right\Vert _{L^{2}\left(
\Omega \right) }\leq C\text{,}
\end{equation*}%
where $C>0$ is a constant independent of $\varepsilon $. Then the trace of $%
u_{\varepsilon }$ on $\partial T^{\varepsilon }$ satisfies%
\begin{equation*}
\varepsilon \int_{\partial T^{\varepsilon }}\left\vert u_{\varepsilon
}\left( x\right) \right\vert ^{2}d\sigma _{\varepsilon }\left( x\right) \leq
C
\end{equation*}%
for all $\varepsilon \in E$, and up to a subsequence, it two-scale converges
in the sens of Definition \ref{def2.2} to some $u_{0}\in L^{2}\left( \Omega
;L_{per}^{2}\left( \partial T\right) \right) $, which is the trace on $%
\partial T$ of a function in $L^{2}\left( \Omega ;H_{\#}^{1}\left( Y\right)
\right) $. More precisely, there exists a subsequence $E^{\prime }$ of $E$
and a function $u_{0}\in L^{2}\left( \Omega ;H_{per}^{1}\left( Y\right)
\right) $ such that as $E^{\prime }\ni \varepsilon \rightarrow 0$,%
\begin{equation*}
\varepsilon \int_{\partial T^{\varepsilon }}u_{\varepsilon }\left( x\right)
\psi ^{\varepsilon }\left( x\right) d\sigma _{\varepsilon }\left( x\right)
\rightarrow \int \int_{\Omega \times \partial T}u_{0}\left( x,y\right) \psi
\left( x,y\right) dxd\sigma \left( y\right) \text{ \ for all }\psi \in 
\mathcal{C}\left( \overline{\Omega };\mathcal{C}_{per}\left( Y\right)
\right) \text{,}
\end{equation*}%
\begin{equation*}
\int_{\Omega }u_{\varepsilon }\left( x\right) \psi ^{\varepsilon }\left(
x\right) dx\rightarrow \int \int_{\Omega \times Y}u_{0}\left( x,y\right)
\psi \left( x,y\right) dxdy\text{ \ for all }\psi \in L^{2}\left( \Omega ;%
\mathcal{C}_{per}\left( Y\right) \right)
\end{equation*}%
and 
\begin{equation*}
\varepsilon \int_{\Omega }\frac{\partial u_{\varepsilon }}{\partial x_{j}}%
\left( x\right) \psi ^{\varepsilon }\left( x\right) dx\rightarrow \int
\int_{\Omega \times Y}\frac{\partial u_{0}}{\partial y_{j}}\left( x,y\right)
\psi \left( x,y\right) dxdy\text{ \ for all }\psi \in L^{2}\left( \Omega ;%
\mathcal{C}_{per}\left( Y\right) \right) \text{,}
\end{equation*}%
for all $1\leq j\leq N$.
\end{proposition}

Having made the above preliminaries, let us turn now to the statement of the
hypotheses for the homogenization problem of (\ref{eq1.3})-(\ref{eq1.7}). In
view of (\ref{eq1.8a}) and since the functions $a_{ij}$, $\mathcal{\theta }$
and $f_{j}$ belong to $L^{\infty }\left( \mathbb{R}^{N}\right) $ we have 
\begin{equation}
a_{ij},\mathcal{\theta }\text{ and }f_{j}\in L_{per}^{\infty }\left(
Y\right) \text{\qquad }\left( 1\leq i,j\leq N\right) \text{.}
\label{eq2.10c}
\end{equation}%
Further, since the sets $k+T$ $\left( k\in \mathbb{Z}^{N}\right) $ are
pairwise disjoint, the characteristic function, $\mathcal{\chi }_{\Theta }$,
of the set $\Theta $ ($\Theta $ is defined in (\ref{eq2.6a})) verifies 
\begin{equation*}
\mathcal{\chi }_{\Theta }=\sum_{k\in \mathbb{Z}^{N}}\mathcal{\chi }_{k+T}%
\text{ \quad a.e. in }\mathbb{R}^{N}\text{,}
\end{equation*}%
where $\mathcal{\chi }_{k+T}$ is the characteristic function of $k+T$ in $%
\mathbb{R}_{y}^{N}$. We have the following proposition.

\begin{proposition}
\label{pr2.3} The characteristic function of the set $\Theta $ ($\Theta $ is
given by (\ref{eq2.6a})), $\mathcal{\chi }_{\Theta }$ belongs to $%
L_{per}^{\infty }\left( Y\right) $ and moreover its mean value is 
\begin{equation*}
\int_{Y}\mathcal{\chi }_{\Theta }\left( y\right) dy=\left\vert T\right\vert 
\text{.}
\end{equation*}
\end{proposition}

Now, we consider the open set 
\begin{equation*}
G=\mathbb{R}_{y}^{N}\backslash \Theta
\end{equation*}%
and its characteristic function $\mathcal{\chi }_{G}$. We have:%
\begin{equation}
\mathcal{\chi }_{G}\in L_{per}^{\infty }\left( Y\right) \text{.}
\label{eq2.10a}
\end{equation}%
Indeed, $\mathcal{\chi }_{G}=1-\mathcal{\chi }_{\Theta }$ and $\mathcal{\chi 
}_{\Theta }\in L_{per}^{\infty }\left( Y\right) $ (in view of Proposition %
\ref{pr2.3}). Further, as $\int_{Y}\mathcal{\chi }_{\Theta }\left( y\right)
dy=\left\vert T\right\vert $ we have 
\begin{equation}
\int_{Y}\mathcal{\chi }_{G}\left( y\right) dy=\left\vert Y\right\vert
-\left\vert T\right\vert =\left\vert Y^{\ast }\right\vert \text{.}
\label{eq2.10b}
\end{equation}%
Moreover, let us notice that 
\begin{equation*}
\mathcal{\chi }_{\Theta }=\sum_{k\in \mathbb{Z}^{N}}\mathcal{\chi }_{k+%
\overset{\circ }{T}}
\end{equation*}%
a.e. in $\mathbb{R}^{N}$ ( $\overset{\circ }{T}$\ being the interior of $T$)
and Therefore $\mathcal{\chi }_{Y}\mathcal{\chi }_{\Theta }=\mathcal{\chi }_{%
\overset{\circ }{T}}=\mathcal{\chi }_{T}$ a.e. in $\mathbb{R}^{N}$. Thus, we
have 
\begin{equation}
\mathcal{\chi }_{Y}\mathcal{\chi }_{G}=\mathcal{\chi }_{Y^{\ast }}
\label{eq2.10d}
\end{equation}%
a.e. in $\mathbb{R}^{N}$.

Let us state the following useful lemma.

\begin{lemma}
\label{lem2.4} Let $E$ be a fundamental sequence. Let $\left( u_{\varepsilon
}\right) _{\varepsilon \in E}\subset L^{2}\left( \Omega \right) $ and $%
\left( v_{\varepsilon }\right) _{\varepsilon \in E}\subset L^{\infty }\left(
\Omega \right) $ be two sequences such that:

(i) $u_{\varepsilon }\rightarrow u_{0}$ in $L^{2}\left( \Omega \right) $%
-weak $\Sigma $ as $E\ni \varepsilon \rightarrow 0$,

(ii) $v_{\varepsilon }\rightarrow v_{0}$ in $L^{2}\left( \Omega \right) $%
-strong $\Sigma $ as $E\ni \varepsilon \rightarrow 0$,

(iii) $\left( v_{\varepsilon }\right) _{\varepsilon \in E}$ is bounded in $%
L^{\infty }\left( \Omega \right) $.

Then $u_{\varepsilon }v_{\varepsilon }\rightarrow u_{0}v_{0}$ in $%
L^{2}\left( \Omega \right) $-weak $\Sigma $ as $E\ni \varepsilon \rightarrow
0$.
\end{lemma}

\begin{proof}
By \cite[Proposition 4.7]{bib11} we see that $u_{\varepsilon }v_{\varepsilon
}\rightarrow u_{0}v_{0}$ in $L^{1}\left( \Omega \right) $-weak $\Sigma $ as $%
E\ni \varepsilon \rightarrow 0$. On the other hand, (i) implies that $\left(
u_{\varepsilon }\right) _{\varepsilon \in E}$ is bounded in $L^{2}\left(
\Omega \right) $. Combining this with (iii) we see that $\left(
u_{\varepsilon }v_{\varepsilon }\right) _{\varepsilon \in E}$ is bounded in $%
L^{2}\left( \Omega \right) $. Hence, by Theorem \ref{th2.1}, we can extract
a subsequence $E^{\prime }$ from $E$ such that $u_{\varepsilon
}v_{\varepsilon }\rightarrow z_{0}$ in $L^{2}\left( \Omega \right) $-weak $%
\Sigma $ as $E^{\prime }\ni \rightarrow 0$, where $z_{0}\in L^{2}\left(
\Omega ;L_{per}^{2}\left( Y\right) \right) $, of course. But on one hand, $%
z_{0}\in L^{1}\left( \Omega ;L_{per}^{1}\left( Y\right) \right) $, on the
other hand, $L^{\infty }\left( \Omega ;\mathcal{C}_{per}\left( Y\right)
\right) \subset L^{2}\left( \Omega ;\mathcal{C}_{per}\left( Y\right) \right) 
$ (we recall that $\Omega $ is bounded). Therefore, according to the
definition of the weak $\Sigma $-convergence in $L^{2}\left( \Omega \right) $%
, it follows that as $E^{\prime }\ni \varepsilon \rightarrow 0$, 
\begin{equation*}
\int_{\Omega }u_{\varepsilon }v_{\varepsilon }\psi ^{\varepsilon
}dx\rightarrow \int \int_{\Omega \times Y}z_{0}\left( x,y\right) \psi \left(
x,y\right) dxdy\qquad \left( \psi \in L^{\infty }\left( \Omega ;\mathcal{C}%
_{per}\left( Y\right) \right) \right)
\end{equation*}%
where $\psi ^{\varepsilon }\left( x\right) =\psi \left( x,\frac{x}{%
\varepsilon }\right) $ \ $\left( x\in \Omega \right) $. This means precisely
that the sequence $\left( u_{\varepsilon }v_{\varepsilon }\right)
_{\varepsilon \in E^{\prime }}$ weakly $\Sigma $-converges in $L^{1}\left(
\Omega \right) $ to $z_{0}$. Hence $z_{0}=u_{0}v_{0}$ (by the unicity of the 
$\Sigma $-limit) and further it is the whole sequence $\left( u_{\varepsilon
}v_{\varepsilon }\right) _{\varepsilon \in E}$ (and not only the extracted
subsequence $\left( u_{\varepsilon }v_{\varepsilon }\right) _{\varepsilon
\in E^{\prime }}$) that weakly $\Sigma $-converges in $L^{2}\left( \Omega
\right) $ to $u_{0}v_{0}$.
\end{proof}

\section{A convergence result for (\protect\ref{eq1.3})-(\protect\ref{eq1.7})%
}

Our aim in the present section is to investigate the asymptotic behaviour,
as $\varepsilon \rightarrow 0$, of $\left( \mathbf{u}_{\varepsilon
},p_{\varepsilon }\right) $ solution to (\ref{eq1.3})-(\ref{eq1.7}). To this
end, let us state some preliminaries.

We have the following proposition on the estimates of solutions to (\ref%
{eq1.3})-(\ref{eq1.7}).

\begin{proposition}
\label{pr3.1} Suppose that (\ref{eq1.2a})-(\ref{eq1.2}) are verified. For $%
0<\varepsilon <1$, let $\left( \mathbf{u}_{\varepsilon },p_{\varepsilon
}\right) $ be the unique solution to (\ref{eq1.3})-(\ref{eq1.7}) and $%
\mathcal{P}^{\varepsilon }$ the extension operator in Proposition \ref{pr2.1}%
. There exists a constant $C>0$ independent of $\varepsilon $ such that 
\begin{equation}
\left\Vert \mathcal{P}^{\varepsilon }\mathbf{u}_{\varepsilon }\right\Vert
_{H_{0}^{1}\left( \Omega \right) ^{N}}\leq C\text{,}  \label{eq2.12a}
\end{equation}%
\begin{equation}
\varepsilon \sum_{k=1}^{N}\int_{\partial T^{\varepsilon }}\left\vert
u_{\varepsilon }^{k}\left( x\right) \right\vert ^{2}d\sigma _{\varepsilon
}\left( x\right) \leq C  \label{eq2.12b}
\end{equation}%
and 
\begin{equation}
\left\Vert p_{\varepsilon }\right\Vert _{L^{2}\left( \Omega ^{\varepsilon
}\right) }\leq C\text{.}  \label{eq2.12c}
\end{equation}
\end{proposition}

\begin{proof}
We turn back to the variational probem (\ref{eq1.8}), we take in particular $%
\mathbf{v}=\mathbf{u}_{\varepsilon }$ in the equation. In virtue of (\ref%
{eq1.2a})-(\ref{eq1.2}), this leads to 
\begin{equation}
\alpha \left\Vert \nabla \mathbf{u}_{\varepsilon }\right\Vert _{L^{2}\left(
\Omega ^{\varepsilon }\right) }^{2}+\varepsilon \alpha
_{0}\sum_{k=1}^{N}\int_{\partial T^{\varepsilon }}\left\vert u_{\varepsilon
}^{k}\left( x\right) \right\vert ^{2}d\sigma _{\varepsilon }\left( x\right)
\leq \left\Vert \mathbf{u}_{\varepsilon }\right\Vert _{L^{2}\left( \Omega
^{\varepsilon }\right) }\left\Vert \mathbf{f}^{\varepsilon }\right\Vert
_{L^{2}\left( \Omega \right) }\text{.}  \label{eq2.13a}
\end{equation}%
But, by Proposition \ref{pr2.1} we have 
\begin{equation*}
\alpha \left\Vert \nabla \mathbf{u}_{\varepsilon }\right\Vert _{L^{2}\left(
\Omega ^{\varepsilon }\right) }^{2}\geq \frac{\alpha }{c}\left\Vert \nabla 
\mathcal{P}^{\varepsilon }\mathbf{u}_{\varepsilon }\right\Vert _{L^{2}\left(
\Omega \right) }^{2}\geq \frac{\alpha }{c^{\prime }c}\left\Vert \mathcal{P}%
^{\varepsilon }\mathbf{u}_{\varepsilon }\right\Vert _{L^{2}\left( \Omega
\right) }^{2}\text{,}
\end{equation*}%
where $c$ is the constant in (\ref{eq2.3}) and $c^{\prime }$ the one in the
Poincar\'{e}'s inequality. Thus, using (\ref{eq2.13a}), one has 
\begin{equation*}
\frac{\alpha }{c^{\prime }c}\left\Vert \mathcal{P}^{\varepsilon }\mathbf{u}%
_{\varepsilon }\right\Vert _{L^{2}\left( \Omega \right) }^{2}\leq \lambda
\left( \Omega \right) ^{\frac{1}{2}}\left\Vert \mathbf{f}\right\Vert
_{\infty }\left\Vert \mathcal{P}^{\varepsilon }\mathbf{u}_{\varepsilon
}\right\Vert _{L^{2}\left( \Omega \right) }\text{.}
\end{equation*}%
It follows by (\ref{eq2.13a}), Proposition \ref{pr2.1} and the preceding
inequality that%
\begin{equation}
\left\Vert \mathcal{P}^{\varepsilon }\mathbf{u}_{\varepsilon }\right\Vert
_{L^{2}\left( \Omega \right) }\leq \frac{c^{\prime }c}{\alpha }\lambda
\left( \Omega \right) ^{\frac{1}{2}}\left\Vert \mathbf{f}\right\Vert
_{\infty }\text{, }\left\Vert \nabla \mathcal{P}^{\varepsilon }\mathbf{u}%
_{\varepsilon }\right\Vert _{L^{2}\left( \Omega \right) }\leq \frac{c\sqrt{%
c^{\prime }}}{\alpha }\lambda \left( \Omega \right) ^{\frac{1}{2}}\left\Vert 
\mathbf{f}\right\Vert _{\infty }\text{ }  \label{eq2.13d}
\end{equation}%
and 
\begin{equation}
\varepsilon \sum_{k=1}^{N}\int_{\partial T^{\varepsilon }}\left\vert
u_{\varepsilon }^{k}\left( x\right) \right\vert ^{2}d\sigma _{\varepsilon
}\left( x\right) \leq \frac{c^{\prime }c}{\alpha _{0}\alpha }\lambda \left(
\Omega \right) \left\Vert \mathbf{f}\right\Vert _{\infty }^{2}\text{.}
\label{eq2.13e}
\end{equation}%
Furthermore $p_{\varepsilon }\in L^{2}\left( \Omega ^{\varepsilon };\mathbb{R%
}\right) \mathfrak{/}\mathbb{R}$ and the extension by zero of $%
p_{\varepsilon }$ in the whole $\Omega $, denoted by $\widetilde{p}%
_{\varepsilon }$ belongs to $L^{2}\left( \Omega ;\mathbb{R}\right) \mathfrak{%
/}\mathbb{R}$. Consequently, in virtue of \cite[p. 30]{bib16} there exists
some $\mathbf{v}_{\varepsilon }\in $ $H_{0}^{1}\left( \Omega ;\mathbb{R}%
\right) ^{N}$ such that%
\begin{equation}
\left\{ 
\begin{array}{c}
div\mathbf{v}_{\varepsilon }=\widetilde{p}_{\varepsilon }\text{,} \\ 
\left\Vert \mathbf{v}_{\varepsilon }\right\Vert _{H^{1}\left( \Omega \right)
^{N}}\leq c_{1}\left\Vert \widetilde{p}_{\varepsilon }\right\Vert
_{L^{2}\left( \Omega \right) }%
\end{array}%
\right.  \label{eq2.13c}
\end{equation}%
where the constant $c_{1}$ depends solely on $\Omega $. Multiplying (\ref%
{eq1.3}) by $\mathbf{v}_{\varepsilon }$, we obtain 
\begin{equation*}
\sum_{i,j,k=1}^{N}\int_{\Omega ^{\varepsilon }}a_{ij}^{\varepsilon }\frac{%
\partial u_{\varepsilon }^{k}}{\partial x_{j}}\frac{\partial v_{\varepsilon
}^{k}}{\partial x_{i}}dx+\int_{\partial T^{\varepsilon }}\left(
p_{\varepsilon }\mathbf{n}-\frac{\partial \mathbf{u}_{\varepsilon }}{%
\partial \mathbf{n}_{P^{\varepsilon }}}\right) {\small \cdot }\mathbf{v}%
_{\varepsilon }d\sigma _{\varepsilon }-\left\Vert p_{\varepsilon
}\right\Vert _{L^{2}\left( \Omega ^{\varepsilon }\right) }^{2}=\int_{\Omega
^{\varepsilon }}\mathbf{f}^{\varepsilon }{\small \cdot }\mathbf{v}%
_{\varepsilon }dx\text{.}
\end{equation*}%
Thus, by (\ref{eq1.7}) and the preceding equality, one has 
\begin{equation}
\left\Vert p_{\varepsilon }\right\Vert _{L^{2}\left( \Omega ^{\varepsilon
}\right) }^{2}=\mathbf{a}_{\varepsilon }\left( \mathbf{u}_{\varepsilon },%
\mathbf{v}_{\varepsilon }\right) -\int_{\Omega ^{\varepsilon }}\mathbf{f}%
^{\varepsilon }{\small \cdot }\mathbf{v}_{\varepsilon }dx\text{.}
\label{eq2.13b}
\end{equation}%
Further, on one hand we 
\begin{equation*}
\left\vert \mathbf{a}_{\varepsilon }\left( \mathbf{u}_{\varepsilon },\mathbf{%
v}_{\varepsilon }\right) \right\vert \leq c^{\prime \prime }\left\Vert
\nabla \mathbf{u}_{\varepsilon }\right\Vert _{L^{2}\left( \Omega
^{\varepsilon }\right) }\left\Vert \nabla \mathbf{v}_{\varepsilon
}\right\Vert _{L^{2}\left( \Omega ^{\varepsilon }\right) }
\end{equation*}%
\begin{equation*}
+c^{\prime \prime \prime }\left( \varepsilon \sum_{k=1}^{N}\int_{\partial
T^{\varepsilon }}\left\vert u_{\varepsilon }^{k}\left( x\right) \right\vert
^{2}d\sigma _{\varepsilon }\left( x\right) \right) ^{\frac{1}{2}}\left(
\varepsilon \sum_{k=1}^{N}\int_{\partial T^{\varepsilon }}\left\vert
v_{\varepsilon }^{k}\left( x\right) \right\vert ^{2}d\sigma _{\varepsilon
}\left( x\right) \right) ^{\frac{1}{2}}\text{,}
\end{equation*}%
where $c^{\prime \prime }=N^{2}\max_{1\leq i,j\leq N}\left\Vert
a_{ij}\right\Vert _{\infty }$ and $c^{\prime \prime \prime }=\left\Vert 
\mathcal{\theta }\right\Vert _{\infty }$. Thus, using (\ref{eq2.13d})-(\ref%
{eq2.13c}), one has 
\begin{equation}
\left\vert \mathbf{a}_{\varepsilon }\left( \mathbf{u}_{\varepsilon },\mathbf{%
v}_{\varepsilon }\right) \right\vert \leq c^{\prime \prime }c_{1}\frac{\sqrt{%
c^{\prime }c}}{\alpha }\lambda \left( \Omega \right) ^{\frac{1}{2}%
}\left\Vert \mathbf{f}\right\Vert _{\infty }\left\Vert p_{\varepsilon
}\right\Vert _{L^{2}\left( \Omega ^{\varepsilon }\right) }+c^{\prime \prime
\prime }\sqrt{\frac{c^{\prime }c}{\alpha _{0}\alpha }}\lambda \left( \Omega
\right) ^{\frac{1}{2}}\left\Vert \mathbf{f}\right\Vert _{\infty }\left(
\varepsilon \sum_{k=1}^{N}\int_{\partial T^{\varepsilon }}\left\vert
v_{\varepsilon }^{k}\left( x\right) \right\vert ^{2}d\sigma _{\varepsilon
}\left( x\right) \right) ^{\frac{1}{2}}\text{.}  \label{eq2.13f}
\end{equation}%
Moreover, by rescaling and summation over the cells $\varepsilon \left(
l+Y^{\ast }\right) \cap \Omega $, with $l\in t^{\varepsilon }$, the trace
inequality in the unit cell yields 
\begin{equation*}
\varepsilon \sum_{k=1}^{N}\int_{\partial T^{\varepsilon }}\left\vert
v_{\varepsilon }^{k}\left( x\right) \right\vert ^{2}d\sigma _{\varepsilon
}\left( x\right) \leq c_{2}\left( \sum_{k=1}^{N}\int_{\Omega ^{\varepsilon
}}\left\vert v_{\varepsilon }^{k}\left( x\right) \right\vert
^{2}dx+\varepsilon ^{2}\sum_{k=1}^{N}\int_{\Omega ^{\varepsilon }}\left\vert
\nabla v_{\varepsilon }^{k}\left( x\right) \right\vert ^{2}dx\right) \text{,}
\end{equation*}%
where $c_{2}>0$ is a constant depending solely on $Y^{\ast }$. Thus, for all 
$0<\varepsilon <1$,%
\begin{equation}
\varepsilon \sum_{k=1}^{N}\int_{\partial T^{\varepsilon }}\left\vert
v_{\varepsilon }^{k}\left( x\right) \right\vert ^{2}d\sigma _{\varepsilon
}\left( x\right) \leq c_{2}\left( \left\Vert \mathbf{v}_{\varepsilon
}\right\Vert _{L^{2}\left( \Omega \right) }^{2}+\varepsilon ^{2}\left\Vert
\nabla \mathbf{v}_{\varepsilon }\right\Vert _{L^{2}\left( \Omega \right)
}^{2}\right) \leq c_{2}\left\Vert \mathbf{v}_{\varepsilon }\right\Vert
_{H^{1}\left( \Omega \right) ^{N}}^{2}\text{.}  \label{eq2.13g}
\end{equation}%
Hence, combining (\ref{eq2.13c}), (\ref{eq2.13f}) and (\ref{eq2.13g}) we
obtain 
\begin{equation*}
\left\vert \mathbf{a}_{\varepsilon }\left( \mathbf{u}_{\varepsilon },\mathbf{%
v}_{\varepsilon }\right) \right\vert \leq c^{\prime \prime }c_{1}\frac{\sqrt{%
c^{\prime }c}}{\alpha }\lambda \left( \Omega \right) ^{\frac{1}{2}%
}\left\Vert \mathbf{f}\right\Vert _{\infty }\left\Vert p_{\varepsilon
}\right\Vert _{L^{2}\left( \Omega ^{\varepsilon }\right) }+c^{\prime \prime
\prime }c_{1}\sqrt{c_{2}}\sqrt{\frac{c^{\prime }c}{\alpha _{0}\alpha }}%
\lambda \left( \Omega \right) ^{\frac{1}{2}}\left\Vert \mathbf{f}\right\Vert
_{\infty }\left\Vert p_{\varepsilon }\right\Vert _{L^{2}\left( \Omega
^{\varepsilon }\right) }\text{, }
\end{equation*}%
i.e., 
\begin{equation}
\left\vert \mathbf{a}_{\varepsilon }\left( \mathbf{u}_{\varepsilon },\mathbf{%
v}_{\varepsilon }\right) \right\vert \leq c_{1}\sqrt{c^{\prime }c}\left( 
\frac{c^{\prime \prime }}{\alpha }+c^{\prime \prime \prime }\frac{\sqrt{c_{2}%
}}{\sqrt{\alpha _{0}\alpha }}\right) \lambda \left( \Omega \right) ^{\frac{1%
}{2}}\left\Vert \mathbf{f}\right\Vert _{\infty }\left\Vert p_{\varepsilon
}\right\Vert _{L^{2}\left( \Omega ^{\varepsilon }\right) }\text{.}
\label{eq2.13h}
\end{equation}%
On the other hand, 
\begin{equation}
\left\vert \int_{\Omega ^{\varepsilon }}\mathbf{f}^{\varepsilon }{\small %
\cdot }\mathbf{v}_{\varepsilon }dx\right\vert \leq \lambda \left( \Omega
\right) ^{\frac{1}{2}}\left\Vert \mathbf{f}\right\Vert _{\infty }\left\Vert 
\mathbf{v}_{\varepsilon }\right\Vert _{L^{2}\left( \Omega \right) }\leq
c_{1}\lambda \left( \Omega \right) ^{\frac{1}{2}}\left\Vert \mathbf{f}%
\right\Vert _{\infty }\left\Vert p_{\varepsilon }\right\Vert _{L^{2}\left(
\Omega ^{\varepsilon }\right) }\text{,}  \label{eq2.13i}
\end{equation}%
$c_{1}$ being the constant in (\ref{eq2.13c}). The inequalities (\ref%
{eq2.13h})-(\ref{eq2.13i}) and the relation (\ref{eq2.13b}) lead to 
\begin{equation}
\left\Vert p_{\varepsilon }\right\Vert _{L^{2}\left( \Omega ^{\varepsilon
}\right) }\leq c_{1}\left( \sqrt{c^{\prime }c}\left( \frac{c^{\prime \prime }%
}{\alpha }+c^{\prime \prime \prime }\frac{\sqrt{c_{2}}}{\sqrt{\alpha
_{0}\alpha }}\right) +1\right) \lambda \left( \Omega \right) ^{\frac{1}{2}%
}\left\Vert \mathbf{f}\right\Vert _{\infty }\text{, }  \label{eq2.13j}
\end{equation}%
for all $0<\varepsilon <1$. In view of (\ref{eq2.13d}), (\ref{eq2.13e}) and (%
\ref{eq2.13j}), there exists a positive constant $C$ independent of $%
\varepsilon $ such that (\ref{eq2.12a})-(\ref{eq2.12c}) are satisfied.
\end{proof}

Before we can establish the so-called global homogenization theorem for (\ref%
{eq1.3})-(\ref{eq1.7}), we require a few basic notation and results. To
begin, let 
\begin{equation*}
\mathcal{V}=\mathcal{D}\left( \Omega ;\mathbb{R}\right) ^{N}\text{,}
\end{equation*}%
\begin{equation*}
\mathbf{V=}H_{0}^{1}\left( \Omega ;\mathbb{R}\right) ^{N}
\end{equation*}%
\begin{equation*}
\mathcal{V}_{Y}=\left\{ \mathbf{\psi }\in \mathcal{C}_{per}^{\infty }\left(
Y;\mathbb{R}\right) ^{N}:\int_{Y}\mathbf{\psi }\left( y\right) dy=0\right\} 
\text{, }
\end{equation*}%
\begin{equation*}
\mathbf{V}_{Y}=H_{\#}^{1}\left( Y;\mathbb{R}\right) ^{N}\text{ }
\end{equation*}%
where: $\mathcal{C}_{per}^{\infty }\left( Y;\mathbb{R}\right) =\mathcal{C}%
^{\infty }\left( \mathbb{R}^{N};\mathbb{R}\right) \cap \mathcal{C}%
_{per}\left( Y\right) $. We provide $\mathbf{V}_{Y}$ with the $%
H_{\#}^{1}\left( Y\right) ^{N}$-norm, which makes it a Hilbert space. There
is no difficulty in verifying that $\mathcal{V}_{Y}$ is dense in $\mathbf{V}%
_{Y}$ . With this in mind, set 
\begin{equation*}
\mathbb{F}_{0}^{1}=\mathbf{V}\times L^{2}\left( \Omega ;\mathbf{V}%
_{Y}\right) \text{.}
\end{equation*}%
This is a Hilbert space with norm 
\begin{equation*}
\left\Vert \mathbf{v}\right\Vert _{\mathbb{F}_{0}^{1}}=\left( \left\Vert 
\mathbf{v}_{0}\right\Vert _{H_{0}^{1}\left( \Omega \right)
^{N}}^{2}+\left\Vert \mathbf{v}_{1}\right\Vert _{L^{2}\left( \Omega
;V_{Y}\right) }^{2}\right) ^{\frac{1}{2}}\text{, }\mathbf{v=}\left( \mathbf{v%
}_{0},\mathbf{v}_{1}\right) \in \mathbb{F}_{0}^{1}\text{.}
\end{equation*}%
On the other hand, put 
\begin{equation*}
\mathbf{\tciFourier }_{0}^{\infty }=\mathcal{V\times }\left[ \mathcal{D}%
\left( \Omega ;\mathbb{R}\right) \otimes \mathcal{V}_{Y}\right] \text{,}
\end{equation*}%
where $\mathcal{D}\left( \Omega ;\mathbb{R}\right) \otimes \mathcal{V}_{Y}$
stands for the space of vector functions $\mathbf{\phi }$ on $\Omega \times 
\mathbb{R}_{y}^{N}$ of the form 
\begin{equation*}
\mathbf{\phi }\left( x,y\right) =\sum_{finite}\varphi _{i}\left( x\right) 
\mathbf{w}_{i}\left( y\right) \text{ }\left( x\in \Omega ,\text{ }y\in 
\mathbb{R}^{N}\right)
\end{equation*}%
with $\varphi _{i}\in \mathcal{D}\left( \Omega ;\mathbb{R}\right) $, $%
\mathbf{w}_{i}\in \mathcal{V}_{Y}$. It is clear that $\mathbf{\tciFourier }%
_{0}^{\infty }$ is dense in $\mathbb{F}_{0}^{1}$.

It is of interest to notice that for $\mathbf{v}=\left( \mathbf{v}_{0},%
\mathbf{v}_{1}\right) \in \mathbb{F}_{0}^{1}$ with $\mathbf{v}_{0}=\left(
v_{0}^{k}\right) _{1\leq k\leq N}$ and $\mathbf{v}_{1}=\left(
v_{1}^{k}\right) _{1\leq k\leq N}$, if we set

\begin{equation*}
\mathbb{D}_{j}\mathbf{v}^{k}=\frac{\partial v_{0}^{k}}{\partial x_{j}}+\frac{%
\partial v_{1}^{k}}{\partial y_{j}}\text{ }\qquad \left( 1\leq j,k\leq
N\right) \text{,}
\end{equation*}%
then we have 
\begin{equation}
\left\Vert \mathbf{v}\right\Vert _{\mathbb{F}_{0}^{1}}=\left(
\sum_{j,k=1}^{N}\left\Vert \mathbb{D}_{j}\mathbf{v}^{k}\right\Vert
_{L^{2}\left( \Omega ;L_{per}^{2}\left( Y\right) \right) }^{2}\right) ^{%
\frac{1}{2}}\qquad \left( \mathbf{v}=\left( v_{0},v_{1}\right) \in \mathbb{F}%
_{0}^{1}\right) \text{.}  \label{eq2.11}
\end{equation}

Now, for $\mathbf{u}=\left( \mathbf{u}_{0},\mathbf{u}_{1}\right) $ and $%
\mathbf{v}=\left( \mathbf{v}_{0},\mathbf{v}_{1}\right) \in \mathbb{F}%
_{0}^{1} $ we set 
\begin{equation*}
\widehat{a}_{\Omega }\left( \mathbf{u},\mathbf{v}\right)
=\sum_{i,j,k=1}^{N}\int \int_{\Omega \times Y^{\ast }}a_{ij}\mathbb{D}_{j}%
\mathbf{u}^{k}\mathbb{D}_{i}\mathbf{v}^{k}dxdy+\int \int_{\Omega \times
\partial T}\mathcal{\theta }\mathbf{u}_{0}{\small \cdot }\mathbf{v}%
_{0}dxd\sigma \text{.}
\end{equation*}%
This defines a bilinear form $\widehat{a}_{\Omega }\left( ,\right) $ on $%
\mathbb{F}_{0}^{1}\times \mathbb{F}_{0}^{1}$ which in view of (\ref{eq1.2a}%
)-(\ref{eq1.2}), is symmetric, positive, continuous and noncoercive. Indeed,
for some $\mathbf{u}=\left( \mathbf{u}_{0},\mathbf{u}_{1}\right) \in \mathbb{%
F}_{0}^{1}$, $\widehat{a}_{\Omega }\left( \mathbf{u},\mathbf{u}\right) =0$
if and only if $\mathbf{u}_{0}=0$ and $\frac{\partial u_{1}^{k}}{\partial
y_{j}}\left( x,y\right) =0$ a.e. in $\left( x,y\right) \in \Omega \times
Y^{\ast }$ ($1\leq j,k\leq N$). Thus, $\mathbf{u}=\left( \mathbf{u}_{0},%
\mathbf{u}_{1}\right) $ is not necessarily the zero function in $\mathbb{F}%
_{0}^{1}$. However, we put 
\begin{equation*}
\mathbf{N}\left( \mathbf{v}\right) =\left( \sum_{j,k=1}^{N}\int \int_{\Omega
\times Y^{\ast }}\left\vert \mathbb{D}_{j}\mathbf{v}^{k}\right\vert
^{2}dxdy+\int_{\Omega }\left\vert \mathbf{v}_{0}\right\vert ^{2}dx\right) ^{%
\frac{1}{2}}
\end{equation*}%
for all $\mathbf{v}=\left( \mathbf{v}_{0},\mathbf{v}_{1}\right) \in \mathbb{F%
}_{0}^{1}$. This defines a seminorm on $\mathbb{F}_{0}^{1}$. Equipped with
the seminorm $\mathbf{N}\left( .\right) $, $\mathbb{F}_{0}^{1}$ is a
pre-Hilbert space which is nonseparated and noncomplete. Further, let us
consider the linear form $\widehat{l}_{\Omega }$ on $\mathbb{F}_{0}^{1}$
defined by 
\begin{equation*}
\widehat{l}_{\Omega }\left( \mathbf{v}\right) =\int \int_{\Omega \times
Y^{\ast }}\mathbf{f}\left( y\right) {\small \cdot }\mathbf{v}_{0}\left(
x\right) dxdy
\end{equation*}%
for all $\mathbf{v}=\left( \mathbf{v}_{0},\mathbf{v}_{1}\right) \in \mathbb{F%
}_{0}^{1}$. The form $\widehat{l}_{\Omega }$ is continuous on $\mathbb{F}%
_{0}^{1}$ for the norm $\left\Vert .\right\Vert _{\mathbb{F}_{0}^{1}}$ and
the seminorm $\mathbf{N}\left( .\right) $. Therefore, we have the following
convergence theorem..

\begin{lemma}
\label{lem3.2} Suppose (\ref{eq1.2a})-(\ref{eq1.2}) and (\ref{eq1.8a}) hold.
There exists $\mathbf{u=}\left( \mathbf{u}_{0},\mathbf{u}_{1}\right) \in 
\mathbb{F}_{0}^{1}$ satisfying the variational problem%
\begin{equation*}
\widehat{a}_{\Omega }\left( \mathbf{u},\mathbf{v}\right) =\widehat{l}%
_{\Omega }\left( \mathbf{v}\right) \text{ \quad for all }\mathbf{v}\in 
\mathbb{F}_{0}^{1}\text{.}
\end{equation*}%
Moreover, $\mathbf{u}_{0}$ is strictly unique and $\mathbf{u}_{1}$ is unique
up to an additive vector function $\mathbf{g=}\left( g^{k}\right) _{1\leq
k\leq N}\in L^{2}\left( \Omega ;V_{Y}\right) $ such that $\frac{\partial
g^{k}}{\partial y_{j}}\left( x,y\right) =0$ almost everywhere in $\left(
x,y\right) \in \Omega \times Y^{\ast }$ $\left( 1\leq j,k\leq N\right) $.
\end{lemma}

\begin{proof}
The proof of this lemma is a simple adaptation of the one in \cite[Lemma 2.5]%
{bib11}. So, for shortness we omit it.
\end{proof}

Now, let us state our convergence theorem.

\begin{theorem}
\label{th3.1} Suppose that the hypotheses (\ref{eq1.2a})-(\ref{eq1.2}) and (%
\ref{eq1.8a}) are satisfied. For $\varepsilon \in E$, let $\left( \mathbf{u}%
_{\varepsilon },p_{\varepsilon }\right) \in \mathbf{V}_{\varepsilon }\times
\left( L^{2}\left( \Omega ^{\varepsilon }\right) \mathfrak{/}\mathbb{R}%
\right) $ be the unique solution to (\ref{eq1.3})-(\ref{eq1.7}) ($E$ being a
fundamental sequence), and let $\mathcal{P}^{\varepsilon }$ be the extension
operator of Proposition \ref{pr2.1}. Then, a subsequence $E^{\prime }$ can
be extracted from $E$ such that as $E^{\prime }\ni \varepsilon \rightarrow 0$%
, 
\begin{equation}
\mathcal{P}^{\varepsilon }\mathbf{u}_{\varepsilon }\rightarrow \mathbf{u}_{0}%
\text{ in }H_{0}^{1}\left( \Omega \right) ^{N}\text{-weak,}  \label{eq2.15a}
\end{equation}%
\begin{equation}
\widetilde{p}_{\varepsilon }\rightarrow p_{0}\text{ in }L^{2}\left( \Omega
\right) \text{-weak,}  \label{eq2.15b}
\end{equation}%
\begin{equation}
\widetilde{p}_{\varepsilon }\rightarrow p_{1}\text{ in }L^{2}\left( \Omega
\right) \text{-weak }\Sigma  \label{eq2.15c}
\end{equation}%
where $\widetilde{p}_{\varepsilon }$ is the extension by zero of $%
p_{\varepsilon }$ in $\Omega $. Moreover, there exists some $\mathbf{u}%
_{1}\in L^{2}\left( \Omega ;\mathbf{V}_{Y}\right) $ such that $\mathbf{u}%
=\left( \mathbf{u}_{0},\mathbf{u}_{1}\right) $, $p_{0}$ and $p_{1}$ verify
the following variational problem: 
\begin{equation}
\left\{ 
\begin{array}{c}
\mathbf{u=}\left( \mathbf{u}_{0},\mathbf{u}_{1}\right) \in \mathbb{F}_{0}^{1}%
\text{; \qquad \qquad \qquad \qquad \qquad \qquad \qquad \qquad \qquad
\qquad \qquad \qquad \qquad \qquad } \\ 
\widehat{a}_{\Omega }\left( \mathbf{u},\mathbf{v}\right) =\widehat{l}%
_{\Omega }\left( \mathbf{v}\right) +\int_{\Omega }p_{0}div\mathbf{v}%
_{0}dx+\int \int_{\Omega \times Y}p_{1}div_{y}\mathbf{v}_{1}dxdy\text{ \ for
all }\mathbf{v=}\left( \mathbf{v}_{0},\mathbf{v}_{1}\right) \in \mathbb{F}%
_{0}^{1}\text{ ,}%
\end{array}%
\right.  \label{eq2.12}
\end{equation}%
where $div_{y}$ denotes the divergence operator in $\mathbb{R}_{y}^{N}$.
\end{theorem}

\begin{proof}
Let $E$ be a fundamental sequence. According to Proposition \ref{pr3.1}, the
sequences $\left( \mathcal{P}^{\varepsilon }\mathbf{u}_{\varepsilon }\right)
_{\varepsilon \in E}$ and $\left( \widetilde{p}_{\varepsilon }\right)
_{\varepsilon \in E}$ are bounded in $H_{0}^{1}\left( \Omega \right) ^{N}$
and $L^{2}\left( \Omega \right) $ respectively in view of (\ref{eq2.12a})
and (\ref{eq2.12c}). Thus, Theorems \ref{th2.1} and \ref{th2.2} yield a
subsequence $E^{\prime }$ extracted from $E$, a vector function $\mathbf{u}%
=\left( \mathbf{u}_{0},\mathbf{u}_{1}\right) \in H_{0}^{1}\left( \Omega ;%
\mathbb{R}\right) ^{N}\times L^{2}\left( \Omega ;H_{\#}^{1}\left( Y;\mathbb{R%
}\right) ^{N}\right) $ and a function $p_{1}\in L^{2}\left( \Omega
;L_{per}^{2}\left( Y;\mathbb{R}\right) \right) $ such that (\ref{eq2.15a})-(%
\ref{eq2.15c}) hold with $p_{0}\left( x\right) =\int_{Y}p_{1}\left(
x,y\right) dy$ a.e. in $x\in \Omega $, and 
\begin{equation}
\frac{\partial \left( \mathcal{P}^{\varepsilon }\mathbf{u}_{\varepsilon
}\right) ^{k}}{\partial x_{j}}\rightarrow \frac{\partial u_{0}^{k}}{\partial
x_{j}}+\frac{\partial u_{1}^{k}}{\partial y_{j}}\quad \text{ in }L^{2}\left(
\Omega \right) \text{-weak }\Sigma  \label{eq2.15d}
\end{equation}%
as $E^{\prime }\ni \varepsilon \rightarrow 0$. Let us check that $\mathbf{u}%
=\left( \mathbf{u}_{0},\mathbf{u}_{1}\right) $, $p_{0}$ and $p_{1}$ verify (%
\ref{eq2.12}). We notice at once that $\mathbf{u}=\left( \mathbf{u}_{0},%
\mathbf{u}_{1}\right) \in \mathbb{F}_{0}^{1}$.\ For each real $\varepsilon
>0 $, let 
\begin{equation*}
\mathbf{\phi }_{\varepsilon }=\mathbf{\phi }_{0}+\varepsilon \mathbf{\phi }%
_{1}^{\varepsilon }\text{ with }\mathbf{\phi }_{0}=\left( \phi
_{0}^{k}\right) _{1\leq k\leq N}\in \mathcal{V}=\mathcal{D}\left( \Omega ;%
\mathbb{R}\right) ^{N}\text{, }\mathbf{\phi }_{1}=\left( \phi
_{1}^{k}\right) _{1\leq k\leq N}\in \mathcal{D}\left( \Omega ;\mathbb{R}%
\right) \otimes \mathcal{V}_{Y}\text{,}
\end{equation*}%
i.e., $\mathbf{\phi }_{\varepsilon }\left( x\right) =\mathbf{\phi }%
_{0}\left( x\right) +\varepsilon \mathbf{\phi }_{1}\left( x,\frac{x}{%
\varepsilon }\right) $ for $x\in \Omega $. Clearly, we have $\mathbf{\phi }%
_{\varepsilon }\in \mathcal{D}\left( \Omega ;\mathbb{R}\right) ^{N}$ and
further, all the functions $\mathbf{\phi }_{\varepsilon }$ $\left(
\varepsilon >0\right) $ have their supports contained in a fixed compact set 
$K\subset \Omega $. Consequently, in virtue of Lemma \ref{lem2.2}, there is
some $\varepsilon _{0}$ with $0<\varepsilon _{0}<1$ such that 
\begin{equation}
\mathbf{\phi }_{\varepsilon }=0\text{ in }\Omega ^{\varepsilon }\backslash
Q^{\varepsilon }\text{\qquad }\left( 0<\varepsilon \leq \varepsilon
_{0}\right) \text{.}  \label{eq2.14a}
\end{equation}%
This being so, taking a scalar product of (\ref{eq1.3}) with $\mathbf{v}=%
\mathbf{\phi }_{\varepsilon }{\small \mid }_{\Omega ^{\varepsilon }}$ ($%
\mathbf{\phi }_{\varepsilon }{\small \mid }_{\Omega ^{\varepsilon }}$ being
de restriction of $\mathbf{\phi }_{\varepsilon }$ to $\Omega ^{\varepsilon }$%
) and using an integration by part lead us to 
\begin{equation}
\mathbf{a}_{\varepsilon }\left( \mathbf{u}_{\varepsilon },\mathbf{\phi }%
_{\varepsilon }{\small \mid }_{\Omega ^{\varepsilon }}\right) -\int_{\Omega
^{\varepsilon }}p_{\varepsilon }div\mathbf{\phi }_{\varepsilon
}dx=\int_{\Omega ^{\varepsilon }}\mathbf{f}^{\varepsilon }{\small \cdot }%
\mathbf{\phi }_{\varepsilon }dx  \label{eq2.13}
\end{equation}%
in virtue of the boundary conditions (\ref{eq1.5})-(\ref{eq1.7}). Further,
by the decomposition $\Omega ^{\varepsilon }=Q^{\varepsilon }\cup \left(
\Omega ^{\varepsilon }{\small \diagdown }Q^{\varepsilon }\right) $ and use
of $Q^{\varepsilon }=\Omega \cap \varepsilon G$, the equality (\ref{eq2.13})
yields 
\begin{equation}
\sum_{i,j,k=1}^{N}\int_{\Omega }a_{ij}^{\varepsilon }\frac{\partial \left( 
\mathcal{P}^{\varepsilon }\mathbf{u}_{\varepsilon }\right) ^{k}}{\partial
x_{j}}\frac{\partial \phi _{\varepsilon }^{k}}{\partial x_{i}}\mathcal{\chi }%
_{G}^{\varepsilon }dx+\varepsilon \int_{\partial T^{\varepsilon }}\mathcal{%
\theta }^{\varepsilon }\mathcal{P}^{\varepsilon }\mathbf{u}_{\varepsilon }%
{\small \cdot }\mathbf{\phi }_{\varepsilon }d\sigma _{\varepsilon
}-\int_{\Omega }\widetilde{p}_{\varepsilon }div\mathbf{\phi }_{\varepsilon
}dx=\int_{\Omega }\mathbf{f}^{\varepsilon }{\small \cdot }\mathbf{\phi }%
_{\varepsilon }\mathcal{\chi }_{G}^{\varepsilon }dx  \label{eq2.14}
\end{equation}%
for all $0<\varepsilon \leq \varepsilon _{0}$, in view of (\ref{eq2.14a}).

Let us pass to the limit in (\ref{eq2.14}) when $E^{\prime }\ni \varepsilon
\rightarrow 0$. First, by (\ref{eq2.10c}) and (\ref{eq2.10a}) we see that $%
a_{ij}\mathcal{\chi }_{G}$ and $f_{j}\mathcal{\chi }_{G}$ belong to $%
L_{per}^{\infty }\left( Y\right) $. Further, for $1\leq i,k\leq N$, the
sequence $\left( \frac{\partial \phi _{\varepsilon }^{k}}{\partial x_{i}}%
\right) _{\varepsilon \in E}$ is bounded in $L^{\infty }\left( \Omega
\right) $ and 
\begin{equation*}
\frac{\partial \phi _{\varepsilon }^{k}}{\partial x_{i}}\rightarrow \mathbb{D%
}_{i}\mathbf{\phi }^{k}=\frac{\partial \phi _{0}^{k}}{\partial x_{i}}+\frac{%
\partial \phi _{1}^{k}}{\partial y_{i}}\text{ in }L^{2}\left( \Omega \right) 
\text{-strong }\Sigma
\end{equation*}%
as $\varepsilon \rightarrow 0$ (see, e.g., \cite[Lemma 2.2]{bib10} for
details). Then, according to (\ref{eq2.15b}) it follows by Lemma \ref{lem2.4}
that 
\begin{equation*}
\frac{\partial \left( \mathcal{P}^{\varepsilon }\mathbf{u}_{\varepsilon
}\right) ^{k}}{\partial x_{j}}\frac{\partial \phi _{\varepsilon }^{k}}{%
\partial x_{i}}\rightarrow \mathbb{D}_{j}\mathbf{u}^{k}\mathbb{D}_{i}\mathbf{%
\phi }^{k}\text{ in }L^{2}\left( \Omega \right) \text{-weak }\Sigma
\end{equation*}%
as $E^{\prime }\ni \varepsilon \rightarrow 0$. Moreover, by Remark \ref%
{rem2.2} we see that 
\begin{equation}
\sum_{i,j,k=1}^{N}\int_{\Omega }a_{ij}^{\varepsilon }\frac{\partial \left( 
\mathcal{P}^{\varepsilon }\mathbf{u}_{\varepsilon }\right) ^{k}}{\partial
x_{j}}\frac{\partial \phi _{\varepsilon }^{k}}{\partial x_{i}}\mathcal{\chi }%
_{G}^{\varepsilon }dx\rightarrow \sum_{i,j,k=1}^{N}\int \int_{\Omega \times
Y}a_{ij}\mathbb{D}_{j}\mathbf{u}^{k}\mathbb{D}_{i}\mathbf{\phi }^{k}\mathcal{%
\chi }_{G}dxdy  \label{eq2.15}
\end{equation}%
as $E^{\prime }\ni \varepsilon \rightarrow 0$. On the other hand, by (\ref%
{eq2.12b}) of Proposition \ref{pr3.1}, Proposition \ref{pr2.2} and (\ref%
{eq2.15a}) the subsequence $E^{\prime }$ can be extracted from $E$ such that 
\begin{equation*}
\varepsilon \int_{\partial T^{\varepsilon }}\mathcal{\theta }^{\varepsilon }%
\mathcal{P}^{\varepsilon }\mathbf{u}_{\varepsilon }{\small \cdot }\mathbf{%
\phi }_{0}d\sigma _{\varepsilon }\rightarrow \int \int_{\Omega \times
\partial T}\mathcal{\theta }\left( y\right) \mathbf{u}_{0}\left( x\right) 
{\small \cdot }\mathbf{\phi }_{0}\left( x\right) dxd\sigma \left( y\right)
\end{equation*}%
as $E^{\prime }\ni \varepsilon \rightarrow 0$, in virtue of Remark \ref%
{rem2.2} since $\mathcal{\theta }\phi _{0}^{k}\in \mathcal{C}\left( 
\overline{\Omega };L_{per}^{\infty }\left( Y\right) \right) $. Furthermore,
by the same argument we see that 
\begin{equation*}
\varepsilon ^{2}\int_{\partial T^{\varepsilon }}\mathcal{\theta }%
^{\varepsilon }\mathcal{P}^{\varepsilon }\mathbf{u}_{\varepsilon }{\small %
\cdot }\mathbf{\phi }_{1}^{\varepsilon }d\sigma _{\varepsilon }\rightarrow 0
\end{equation*}%
as $E^{\prime }\ni \varepsilon \rightarrow 0$, since $\mathcal{\theta }\phi
_{1}^{k}\in \mathcal{C}\left( \overline{\Omega };L_{per}^{\infty }\left(
Y\right) \right) $. Thus, as $E^{\prime }\ni \varepsilon \rightarrow 0$, 
\begin{equation}
\varepsilon \int_{\partial T^{\varepsilon }}\mathcal{\theta }^{\varepsilon }%
\mathcal{P}^{\varepsilon }\mathbf{u}_{\varepsilon }{\small \cdot }\mathbf{%
\phi }_{\varepsilon }d\sigma _{\varepsilon }\rightarrow \int \int_{\Omega
\times \partial T}\mathcal{\theta }\left( y\right) \mathbf{u}_{0}\left(
x\right) {\small \cdot }\mathbf{\phi }_{0}\left( x\right) dxd\sigma \left(
y\right) \text{.}  \label{eq2.16}
\end{equation}%
Once more, we use Remark \ref{rem2.2} to have 
\begin{equation}
\int_{\Omega }\mathbf{f}^{\varepsilon }{\small \cdot }\mathbf{\phi }%
_{\varepsilon }\mathcal{\chi }_{G}^{\varepsilon }dx\rightarrow \int
\int_{\Omega \times Y}\mathbf{f}{\small \cdot }\mathbf{\phi }_{0}\mathcal{%
\chi }_{G}dxdy\text{, }  \label{eq2.17}
\end{equation}%
as $E^{\prime }\ni \varepsilon \rightarrow 0$. Further, according to (\ref%
{eq2.15c}) as $E^{\prime }\ni \varepsilon \rightarrow 0$ we have 
\begin{equation*}
\int_{\Omega }\widetilde{p}_{\varepsilon }div\mathbf{\phi }_{\varepsilon
}dx\rightarrow \int \int_{\Omega \times Y}p_{1}\left( div\mathbf{\phi }%
_{0}+div_{y}\mathbf{\phi }_{1}\right) dxdy\text{. }
\end{equation*}%
Finally, we pass to the limit in (\ref{eq2.14}) as $E^{\prime }\ni
\varepsilon \rightarrow 0$ and we obtain by (\ref{eq2.15})-(\ref{eq2.17})
and (\ref{eq2.10d}), 
\begin{equation}
\widehat{a}_{\Omega }\left( \mathbf{u},\mathbf{\phi }\right) -\int_{\Omega
}p_{0}div\mathbf{\phi }_{0}dx+\int \int_{\Omega \times Y}p_{1}div_{y}\mathbf{%
\phi }_{1}dxdy=\widehat{l}_{\Omega }\left( \mathbf{\phi }\right) \text{,}
\label{eq2.18}
\end{equation}%
for all $\mathbf{\phi }=\left( \mathbf{\phi }_{0},\mathbf{\phi }_{1}\right)
\in \mathbf{\tciFourier }_{0}^{\infty }$. Thus, using the density of $%
\mathbf{\tciFourier }_{0}^{\infty }$ in $\mathbb{F}_{0}^{1}$ and the
continuity of the forms $\widehat{a}_{\Omega }\left( ,\right) $ and $%
\widehat{l}_{\Omega }$, we see that $\mathbf{u}=\left( \mathbf{u}_{0},%
\mathbf{u}_{1}\right) $, $p_{0}$ and $p_{1}$ verify (\ref{eq2.12}). The
proof of the theorem is complete.
\end{proof}

Now, we wish to give a simple representation of the vector function $\mathbf{%
u}_{1}$ in Theorem \ref{th3.1} (or Lemma \ref{lem3.2}) for further needs.
For this purpose we introduce the bilinear form $\widehat{a}$ on $\mathbf{V}%
_{Y}\times \mathbf{V}_{Y}$ defined by%
\begin{equation*}
\widehat{a}\left( \mathbf{v},\mathbf{w}\right)
=\sum_{i,j,k=1}^{N}\int_{Y^{\ast }}a_{ij}\frac{\partial v^{k}}{\partial y_{j}%
}\frac{\partial w^{k}}{\partial y_{i}}dy
\end{equation*}%
for $\mathbf{v}=\left( v^{k}\right) $ and $\mathbf{w}=\left( w^{k}\right) $
in $\mathbf{V}_{Y}$. Next, for each couple of indices $1\leq i,k\leq N$, we
consider the variational problem%
\begin{equation}
\left\{ 
\begin{array}{c}
\mathbf{\chi }_{ik}\in \mathbf{V}_{Y}:\qquad \qquad \qquad \qquad \\ 
\widehat{a}\left( \mathbf{\chi }_{ik},\mathbf{w}\right)
=\sum_{l=1}^{N}\int_{Y}a_{li}\frac{\partial w^{k}}{\partial y_{l}}dy \\ 
\text{for all }\mathbf{w}=\left( w^{k}\right) \ \text{in }\mathbf{V}_{Y}%
\text{,\qquad }%
\end{array}%
\right.  \label{eq2.19}
\end{equation}%
which admits a solution $\mathbf{\chi }_{ik}$, unique up to an additive
vector function $\mathbf{g}=\left( g^{k}\right) \in \mathbf{V}_{Y}$ such
that $\frac{\partial g^{k}}{\partial y_{j}}=0$ a.e. in $Y^{\ast }$.

\begin{lemma}
\label{lem3.3} Under the hypotheses and notations of Theorem \ref{th3.1},
there is some vector function $\mathbf{g}=\left( g^{k}\right) \in
L^{2}\left( \Omega ;\mathbf{V}_{Y}\right) $ such that $\frac{\partial g^{k}}{%
\partial y_{j}}=0$ a.e. in $\Omega \times Y^{\ast }$ and 
\begin{equation}
\mathbf{u}_{1}\left( x,y\right) =-\sum_{i,k=1}^{N}\frac{\partial u_{0}^{k}}{%
\partial x_{i}}\left( x\right) \mathbf{\chi }_{ik}\left( y\right) +\mathbf{g}%
\left( x,y\right)  \label{eq2.20a}
\end{equation}%
almost everywhere in $\left( x,y\right) \in \Omega \times \mathbb{R}^{N}$.
\end{lemma}

\begin{proof}
In (\ref{eq2.12}), choose the test functions $\mathbf{v}=\left( \mathbf{v}%
_{0},\mathbf{v}_{1}\right) $ such that $\mathbf{v}_{0}=0$, $\mathbf{v}%
_{1}\left( x,y\right) =\varphi \left( x\right) \mathbf{w}\left( y\right) $
for $\left( x,y\right) \in \Omega \times \mathbb{R}^{N}$, where $\varphi \in 
\mathcal{D}\left( \Omega ;\mathbb{R}\right) $ and $\mathbf{w}\in \mathbf{V}%
_{Y}$. Then, almost everywhere in $x\in \Omega $, we have 
\begin{equation}
\left\{ 
\begin{array}{c}
\widehat{a}\left( \mathbf{u}_{1}\left( x,.\right) ,\mathbf{w}\right)
=-\sum_{l,j,k=1}^{N}\frac{\partial u_{0}^{k}}{\partial x_{j}}\left( x\right)
\int_{Y^{\ast }}a_{lj}\frac{\partial w^{k}}{\partial y_{l}}dy \\ 
\text{for all }\mathbf{w}=\left( w^{k}\right) \in \mathbf{V}_{Y}\text{%
.\qquad \qquad \qquad \qquad \qquad \quad }%
\end{array}%
\right.  \label{eq2.20}
\end{equation}%
But it is clear that up to and additive function $\mathbf{g}\left( x\right)
=\left( g^{k}\left( x\right) \right) \in \mathbf{V}_{Y}$ such that $\frac{%
\partial g^{k}\left( x\right) }{\partial y_{j}}=0$ a.e. in $Y^{\ast }$, $%
\mathbf{u}_{1}\left( x,.\right) $ (for fixed $x\in \Omega $) is the sole
function in $\mathbf{V}_{Y}$ solving the variational equation (\ref{eq2.20}%
). On the other hand, it is an easy matter to check that the function of $y$
on the right of (\ref{eq2.20a}) solves the same variational problem. Hence
the lemma follows immediately.
\end{proof}

Now, let us state the so-called macroscopic homogenized equations for (\ref%
{eq1.3})-(\ref{eq1.7}). Our goal here is to derive a well-posed boundary
value problem for $\left( \mathbf{u}_{0},p_{0}\right) $. To begin, for $%
1\leq i,j,k,h\leq N$, let 
\begin{equation}
q_{ijkh}=\delta _{kh}\int_{Y^{\ast }}a_{ij}\left( y\right)
dy-\sum_{l=1}^{N}\int_{Y^{\ast }}a_{il}\left( y\right) \frac{\partial 
\mathcal{\chi }_{jh}^{k}}{\partial y_{l}}\left( y\right) dy\text{,}
\label{eq2.21a}
\end{equation}%
where: $\delta _{kh}$ is the Kronecker symbol, $\mathbf{\chi }_{jh}=\left( 
\mathcal{\chi }_{jh}^{k}\right) $ is defined exactly as in (\ref{eq2.19}).
To the coefficients $q_{ijkh}$ we attach the differential operator $\mathcal{%
Q}$ on $\Omega $ mapping $\mathcal{D}^{\prime }\left( \Omega \right) ^{N}$
into $\mathcal{D}^{\prime }\left( \Omega \right) ^{N}$ ($\mathcal{D}^{\prime
}\left( \Omega \right) $ is the usual space of complex distributions on $%
\Omega $) as 
\begin{equation*}
\left\{ 
\begin{array}{c}
\left( \mathcal{Q}\mathbf{z}\right) ^{k}=-\sum_{i,j,h=1}^{N}q_{ijkh}\frac{%
\partial ^{2}z^{h}}{\partial x_{i}\partial x_{^{j}}}\text{ }\left( 1\leq
k\leq N\right) \quad \\ 
\text{for }\mathbf{z=}\left( z^{h}\right) \text{, }z^{h}\in \mathcal{D}%
^{\prime }\left( \Omega \right) \text{.\qquad \qquad \qquad \qquad \quad }%
\end{array}%
\right.
\end{equation*}%
$\mathcal{Q}$ is the so-called homogenized operator associated to $%
P^{\varepsilon }$ $\left( 0<\varepsilon <1\right) $.

We consider now the boundary value problem 
\begin{equation}
\quad \mathcal{Q}\mathbf{u}_{0}+\widetilde{\mathcal{\theta }}\mathbf{u}_{0}+%
\mathbf{grad}p_{0}=\widetilde{\mathbf{f}}\text{ in }\Omega \text{,}
\label{eq2.21}
\end{equation}%
\begin{equation}
\qquad \qquad \qquad \qquad \qquad \mathbf{u}_{0}=0\text{ on }\partial
\Omega   \label{eq2.23}
\end{equation}%
where $\widetilde{\mathcal{\theta }}=\int_{\partial T}\mathcal{\theta }%
\left( y\right) d\sigma \left( y\right) $ and $\widetilde{\mathbf{f}}%
=\int_{Y^{\ast }}\mathbf{f}\left( y\right) dy$.

\begin{lemma}
\label{lem3.4} Under the hypotheses of Theorem \ref{th3.1}, the boundary
value problem (\ref{eq2.21})-(\ref{eq2.23}) admits at most one weak solution 
$\left( \mathbf{u}_{0},p_{0}\right) $ with $\mathbf{u}_{0}\in
H_{0}^{1}\left( \Omega ;\mathbb{R}\right) ^{N}$, $p_{0}\in L^{2}\left(
\Omega ;\mathbb{R}\right) \mathfrak{/}\mathbb{R}$.
\end{lemma}

\begin{proof}
There is no difficulty to verify that if a couple $\left( \mathbf{u}%
_{0},p_{0}\right) \in H_{0}^{1}\left( \Omega ;\mathbb{R}\right) ^{N}\times
L^{2}\left( \Omega ;\mathbb{R}\right) \mathfrak{\ }$satisfies (\ref{eq2.21}%
)-(\ref{eq2.23}), then the vector function $\mathbf{u}=\left( \mathbf{u}_{0},%
\mathbf{u}_{1}\right) $ [with $\mathbf{u}_{1}$ given by (\ref{eq2.20a})]
satisfies (\ref{eq2.12}). Hence the unicity in (\ref{eq2.21})-(\ref{eq2.23})
follows by Lemma \ref{lem3.2}.
\end{proof}

This leads us to the following theorem.

\begin{theorem}
\label{th3.2} Suppose that the hypotheses (\ref{eq1.2a})-(\ref{eq1.2}) and (%
\ref{eq1.8a}) are satisfied. For each real $0<\varepsilon <1$, let $\left( 
\mathbf{u}_{\varepsilon },p_{\varepsilon }\right) \in \mathbf{V}%
_{\varepsilon }\times \left( L^{2}\left( \Omega ^{\varepsilon };\mathbb{R}%
\right) \mathfrak{/}\mathbb{R}\right) $ be defined by (\ref{eq1.3})-(\ref%
{eq1.7}) and let $\mathcal{P}^{\varepsilon }$ be the extension operator of
Proposition \ref{pr2.1}. Then, as $\varepsilon \rightarrow 0$, we have $%
\mathcal{P}^{\varepsilon }\mathbf{u}_{\varepsilon }\rightarrow \mathbf{u}%
_{0} $ in $H_{0}^{1}\left( \Omega \right) ^{N}$-weak and $\widetilde{p}%
_{\varepsilon }\rightarrow p_{0}$ in $L^{2}\left( \Omega \right) $-weak ($%
\widetilde{p}_{\varepsilon }$ being the extension of $p_{\varepsilon }$ by
zero to $\Omega $ ), where the couple $\left( \mathbf{u}_{0},p_{0}\right) $
lies in $H_{0}^{1}\left( \Omega ;\mathbb{R}\right) ^{N}\times \left(
L^{2}\left( \Omega ;\mathbb{R}\right) \mathfrak{/}\mathbb{R}\right) $ and is
the unique weak solution to (\ref{eq2.21})-(\ref{eq2.23}).
\end{theorem}

\begin{proof}
In view of the proof of Theorem \ref{th3.1}, from any given fundamental
sequence $E$ one can extract a subsequence $E^{\prime }$ such that as $%
E^{\prime }\ni \varepsilon \rightarrow 0$, we have (\ref{eq2.15a})-(\ref%
{eq2.15c}) and $\widetilde{p}_{\varepsilon }\rightarrow p_{0}$ in $%
L^{2}\left( \Omega \right) $-weak. Further (\ref{eq2.18}) holds for all $%
\mathbf{\phi }\mathbf{=}\left( \mathbf{\phi }_{0},\mathbf{\phi }_{1}\right)
\in \mathcal{D}\left( \Omega ;\mathbb{R}\right) ^{N}\times \left[ \mathcal{D}%
\left( \Omega ;\mathbb{R}\right) \otimes \mathcal{V}_{Y}\right] $, with $%
\mathbf{u}=\left( \mathbf{u}_{0},\mathbf{u}_{1}\right) \in \mathbb{F}_{0}^{1}
$. Now, substituting (\ref{eq2.20a}) in (\ref{eq2.18}) and then choosing
therein the $\mathbf{\phi }$'s such that $\mathbf{\phi }_{1}=0$, a simple
computation leads to (\ref{eq2.21}) with evidently (\ref{eq2.23}). Hence the
theorem follows by Lemma \ref{lem3.4} and use of an obvious argument.
\end{proof}

\noindent \textbf{Acknowledgements }\textit{This work is supported by the
national allocation for research modernisation of the Cameroon government.}

\end{document}